\documentclass[a4paper,11pt]{amsart}
\usepackage[utf8]{inputenc}
\usepackage{epsfig,graphics}
\usepackage[all]{xy}
\usepackage{array}
\usepackage{amsthm,amssymb,amsbsy,bbm,a4wide,verbatim,url}
\usepackage{rotating}
\usepackage{color,fdsymbol}

\usepackage{fancybox}
\usepackage{enumitem}

\usepackage{tikz}

\usepackage{amsrefs}
\usepackage{hyperref}

\def\into{\hookrightarrow}
\def\onto{\twoheadrightarrow}
\def\simto{\overset{\sim}{\to}}

\def\kk{\mathbbm{k}}

\def\C{\ensuremath{\mathbbm{C}}}

\def\Z{\mathbbm{Z}}

\def\End{\mathrm{End}}
\def\Hom{\mathrm{Hom}}

\def\GL{\mathrm{GL}}

\newtheorem{theorem}{Theorem}[section]
\newtheorem{lemma}[theorem]{Lemma}

\newtheorem{remark}[theorem]{Remark}

\newtheorem{corollary}[theorem]{Corollary}

\newtheorem{example}[theorem]{Example\/}

\newtheorem{proposition}[theorem]{Proposition}

\numberwithin{equation}{section}

\date{November 24, 2020}

\title[Braid groups of normalizers of reflection subgroups]{Braid groups of normalizers of reflection subgroups}

\author{Thomas Gobet}
\address{Institut Denis Poisson, CNRS UMR 7350\\
Facult\'e des Sciences et Techniques\\
Universit\'e de Tours\\
Parc de Grandmont\\
37200 Tours, France}
\email{thomas.gobet@lmpt.univ-tours.fr}

\author{Anthony Henderson}
\address{Sydney Mathematical Research Institute A14\\
University of Sydney\\
NSW 2006, Australia}
\email{anthony.henderson@sydney.edu.au}

\author{Ivan Marin}
\address{Laboratoire Ami\'enois de Math\'ematique Fondamentale et Appliqu\'ee, CNRS UMR 7352\\
Universit\'e de Picardie Jules Verne\\
33 rue Saint Leu\\
80039 Amiens, France}
\email{ivan.marin@u-picardie.fr}

\thanks{This research was supported by Australian Research Council grant DP170101579.}

\begin{document}

\begin{abstract}
Let $W_0$ be a reflection subgroup of a finite complex reflection group $W$, and let $B_0$ and $B$ be their respective braid groups. In order to construct a Hecke algebra $\widetilde{H}_0$ for the normalizer $N_W(W_0)$, one first considers a natural subquotient $\widetilde{B}_0$ of $B$ which is an extension of $N_W(W_0)/W_0$ by $B_0$. We prove that this extension is split when $W$ is a Coxeter group, and deduce a standard basis for the Hecke algebra $\widetilde{H}_0$. We also give classes of both split and non-split examples in the non-Coxeter case. 
\end{abstract}

\maketitle

\tableofcontents

\section{Introduction}

Let $W$ be a finite complex reflection group and $W_0$ a reflection subgroup of $W$. We write $N_W(W_0)$ for the normalizer of $W_0$ in $W$. There are various cases in which $N_W(W_0)$ is a semidirect product of $W_0$ and some complementary subgroup, i.e.\ there is a known splitting of the short exact sequence of groups
\begin{equation} \label{eq:ses-triv}
1 \to W_0 \to N_W(W_0) \to N_W(W_0)/W_0 \to 1.
\end{equation}
For one such case: when $W$ is a finite Coxeter group, a choice of simple system for $W$ determines a complement of $W_0$ in $N_W(W_0)$, as observed by Howlett~\cite{HOWLETT} and recalled in greater generality in Lemma~\ref{lem:semi-direct} below. For another: when $W_0$ is a parabolic subgroup of $W$, then it always has a complement in $N_W(W_0)$, as shown by Muraleedaran and Taylor~\cite{TAYLORNORM}. 
On the other hand, there are cases where no complement exists, i.e.\ the short exact sequence~\eqref{eq:ses-triv} does not split: see Section \ref{ss:no-complement}.

Let $B$ be the braid group associated to the complex reflection group $W$, defined topologically as in~\cite{BMR}. We can identify the braid group $B_0$ of $W_0$ with a subquotient of $B$.
In~\cite[Section 2.2]{YH2} the third author introduced another subquotient $\widetilde{B}_0$ of $B$, which can be thought of loosely as the braid group of $N_W(W_0)$, although it actually depends on the pair $(W,W_0)$. (The notation $\widetilde{B}_0$ is new to this paper, and refers to the $G=N_W(W_0)$ special case of the group denoted $B_G$ in \textit{loc.\ cit}.) The definition of $\widetilde{B}_0$, recalled in Section~\ref{sec:definitions} below, is such that we have a natural short exact sequence of groups
\begin{equation} \label{eq:ses}
1 \to B_0 \to \widetilde{B}_0 \to N_W(W_0)/W_0 \to 1,
\end{equation}
lifting the short exact sequence~\eqref{eq:ses-triv}.

The main question addressed in this paper is: when can we write $\widetilde{B}_0$ as a semidirect product of $B_0$ and some complementary subgroup? More precisely, assuming we are in a case where we have a splitting of~\eqref{eq:ses-triv}, does that splitting lift to a splitting of~\eqref{eq:ses}? 

One main reason for considering these questions, which was in fact the original motivation, is the study of the Hecke algebra $\widetilde{H}_0$ associated to $N_W(W_0)$, which was defined in~\cite{YH2} as a certain quotient of the group algebra of $\widetilde{B}_0$. A splitting of~\eqref{eq:ses} implies a semidirect product decomposition of $\widetilde{H}_0$. These algebras $\widetilde{H}_0$ are the building blocks
of the algebra $\mathcal{C}_W$ constructed in \cite{YH1} to describe the `Artin part' of the Yokonuma--Hecke algebra, in the sense that $\mathcal{C}_W$ is Morita-equivalent to a direct sum of such Hecke algebras $\widetilde{H}_0$. As explained in~\cite{YH1}, when $W$ is the symmetric group, the algebra $\mathcal{C}_W$ coincides with the diagram algebra of braids and ties of Aicardi and Juyumaya.

In Section~\ref{sec:coxeter} we will show that when $W$ is a finite Coxeter group and $W_0$ is an arbitrary reflection subgroup, the known splitting of~\eqref{eq:ses-triv} does lift to a splitting of~\eqref{eq:ses}; see Theorem~\ref{thm:splitfiniteCox}. 
In Section~\ref{ss:std} we use this to define a standard basis, and a presentation, of $\widetilde{H}_0$ in this case. Our proof of the splitting of~\eqref{eq:ses} applies also when $W_0$ is a reflection subgroup of an infinite Coxeter group $W$, on the assumption that the Artin group of $W_0$ occurs as a subquotient of the Artin group of $W$ in the same manner as in the finite case; see~\eqref{eq:assumption} for the precise statement of this assumption.    

In Section~\ref{sec:groupoids} we explain an alternative proof of the splitting of~\eqref{eq:ses} in the Coxeter case, which is in some ways more conceptual; see Theorem~\ref{thm:splitCox-groupoid}. One aspect of this second proof may be of independent interest: in Section~\ref{ss:groupoids} we give a groupoid description of the complement of $W_0$ in $N_W(W_0)$ when $W$ is a (possibly infinite) Coxeter group and $W_0$ is an arbitrary reflection subgroup, which was inspired by, but is different from, the description given by Brink and Howlett~\cite{BRINKHOWLETT} in the case where $W_0$ is parabolic (see Remark~\ref{rem:brinkhowlett} for a comparison).

The proofs we give in Sections~\ref{sec:coxeter} and~\ref{sec:groupoids} are both intrinsically Coxeter-theoretic, which suggests that, when we return to the setting of general complex reflection groups $W$, the splitting of~\eqref{eq:ses} can most reasonably be expected in those cases which are most Coxeter-like. In Section~\ref{sec:gd1n} we will show that when $W=G(d,1,n)$ and $W_0=G(d,1,k)$, the obvious splitting of~\eqref{eq:ses-triv} does lift to a splitting of~\eqref{eq:ses}. 
On the other hand, in Section~\ref{ss:central} we will give examples where~\eqref{eq:ses-triv} splits but~\eqref{eq:ses} does not split.

\textbf{Acknowledgments}. We are grateful to Bob Howlett, Steen Ryom-Hansen, Mario Salvetti, and Don Taylor for helpful conversations, and to an anonymous referee for their suggestions.


\section{Definitions and preliminaries}
\label{sec:definitions}

The goal of this section is to recall the definitions referred to in the introduction, in particular of the group $\widetilde{B}_0$, the short exact sequence~\eqref{eq:ses}, and the Hecke algebra $\widetilde{H}_0$. For more details, see~\cite{BMR,YH2}.  

Let $W< \GL_n(\C)$ be a finite complex reflection group, let $\mathcal{A}$ denote the arrangement of reflecting hyperplanes in $\C^n$ for the reflections in $W$, and let $X= \C^n \setminus \bigcup_{H\in\mathcal{A}} H$ be the complement of that arrangement, on which $W$ acts freely. We fix a base-point $\tilde{x}\in X$, let $[\tilde{x}]_W$ denote its image in the quotient $X/W$, and define the pure braid group $P=\pi_1(X,\tilde{x})$ and braid group $B=\pi_1(X/W,[\tilde{x}]_W)$. We denote by $\pi : B \onto W$ the natural projection, whose kernel is identified with $P$.

Recall from~\cite[Theorem 2.17(1)]{BMR} that $B$ is generated by the elements known as \emph{braided reflections} 
around the hyperplanes in $\mathcal{A}$. For $H\in\mathcal{A}$, let $m_H$ denote the order of the cyclic subgroup of $W$ fixing $H$, let $s_H\in W$ denote the distinguished reflection with hyperplane $H$, i.e.\ the one with determinant $\exp(2\pi\sqrt{-1}/m_H)$, and let $\sigma_H\in B$ be a braided reflection around $H$ such that $\pi(\sigma_H)=s_H$, as in~\cite[Lemma 2.14]{BMR}. Such a braided reflection $\sigma_H$ is unique up to $P$-conjugacy; more generally, if $\beta\in B$, then $\beta\sigma_H\beta^{-1}$ is a braided reflection around the hyperplane $\pi(\beta)(H)$. The element $\sigma_H^{m_H}\in P$ and its $P$-conjugates are the homotopy classes of the particular loops in $X$ based at $\tilde{x}$ which are known as \emph{meridians} around $H$. By~\cite[Theorem 2.18(1)]{BMR}, $P$ is generated by the set of all the meridians around hyperplanes in $\mathcal{A}$. In fact, by~\cite[Proposition 2.8]{BMR}, it suffices to take one (well-chosen) meridian per hyperplane.

Now let $W_0$ be a reflection subgroup of $W$, let $\mathcal{A}_0\subseteq\mathcal{A}$ be the collection of reflecting hyperplanes of $W_0$, and let $X^0=\C^n \setminus \bigcup_{H\in\mathcal{A}_0} H$. Again we have the pure braid group $P_0=\pi_1(X^0,\tilde{x})$ and braid group $B_0=\pi_1(X^0/W_0,[\tilde{x}]_{W_0})$, and the projection $\pi_0:B_0\onto W_0$ with kernel $P_0$.

The inclusion of $X$ in $X^0$ induces a surjection $P\onto P_0$, whose kernel is the subgroup $K_0$ of $P$ generated by meridians around the hyperplanes in $\mathcal{A}\setminus\mathcal{A}_0$. As explained in~\cite[Section 2.2]{YH2}, we can identify $\pi_1(X/W_0,[\tilde{x}]_{W_0})$ with the subgroup $\pi^{-1}(W_0)$ of $B$, and thus the surjection $P\onto P_0$ extends to a surjection $\pi^{-1}(W_0)\onto B_0$ which still has kernel $K_0$. Hence $B_0$ can be identified with the subquotient $\pi^{-1}(W_0)/K_0$ of $B$, in such a way that $\pi:\pi^{-1}(W_0)\onto W_0$ factors through $\pi_0:B_0\onto W_0$. 

In the case when $W_0$ is a \emph{parabolic} subgroup of $W$, i.e.\ $W_0$ is the pointwise stabilizer in $W$ of some subspace of $\C^n$, it is shown in~\cite[Proposition 2.29]{BMR} that there is a splitting of the surjection $\pi^{-1}(W_0)\onto B_0$, well-defined up to conjugacy by $P$ and compatible with $\pi$ and $\pi_0$. Hence in this case we have a commutative diagram
\begin{equation}
\label{eq:parabolic}
\xymatrix{
 1 \ar[r] & P_0 \ar[r]\ar@{_{(}->}[d] & B_0 \ar[r]^{\pi_0}\ar@{_{(}->}[d] & W_0 \ar[r]\ar@{_{(}->}[d] & 1\\
 1 \ar[r] & P \ar[r] & B \ar[r]^{\pi} & W \ar[r] & 1
}
\end{equation} 
and we can regard $B_0$ as a subgroup of $B$ rather than a subquotient. However, for non-parabolic reflection subgroups $W_0$, the surjection $\pi^{-1}(W_0)\onto B_0$ is not split in general.

Let $N_W(W_0)$ denote the normalizer of $W_0$ in $W$, and define $\widehat{B}_0 =  \pi^{-1}(N_W(W_0)) \leq B$. Let $\sigma_H$ be a braided reflection around a hyperplane $H\in\mathcal{A}\setminus\mathcal{A}_0$ and let $\beta\in \widehat{B}_0$. Then $\sigma_H^{m_H}$ is a generator of $K_0$, and as $\pi(\beta)$ normalizes $W_0$ we have $\pi(\beta)(H)\notin \mathcal{A}_0$, hence $\beta \sigma_H^{m_H} \beta^{-1}\in K_0$. This shows that $K_0$ is still normal in $\widehat{B}_0$. The group $\widetilde{B}_0$ mentioned in the introduction is defined to be the quotient $\widehat{B}_0/K_0$. Note that $\widetilde{B}_0$ contains $B_0=\pi^{-1}(W_0)/K_0$ as a subgroup. 

Let $\widetilde{\pi}_0:\widetilde{B}_0\onto N_W(W_0)$ be the projection induced by $\pi$. Then we have a commutative diagram
\begin{equation}
\label{eq:comm-diag}
\xymatrix{
 1 \ar[r] & B_0 \ar[r]\ar[d]^{\pi_0} & \widetilde{B}_0 \ar[r]\ar[d]^{\widetilde{\pi}_0} & N_W(W_0)/W_0 \ar[r]\ar[d]^{=} & 1\\
 1 \ar[r] & W_0 \ar[r] & N_W(W_0) \ar[r] & N_W(W_0)/W_0 \ar[r] & 1}
\end{equation}
in which both rows are short exact sequences. These are the short exact sequences~\eqref{eq:ses-triv} and ~\eqref{eq:ses} mentioned in the introduction. It is trivial that any splitting of the top row would induce a splitting of the bottom row. In this paper, we consider cases where we have a splitting of the bottom row (equivalently, we have a subgroup of $N_W(W_0)$ complementary to $W_0$) and investigate whether it lifts to a splitting of the top row.

Recall from~\cite{BMR} the definition of the Hecke algebra $H_0$ associated to $W_0$, which is a quotient of the group algebra $\kk B_0$ by certain Hecke relations. Here $\kk$ can be taken to be the generic ring $\Z[a_{H,i},a_{H,0}^{\pm 1}]$ where $a_{H,i}$ are indeterminates indexed by $W_0$-orbits of hyperplanes $H\in\mathcal{A}_0$ and integers $0\leq i<m_H$.

Recall from \cite{YH2} that the Hecke algebra $\widetilde{H}_0$ associated to $N_W(W_0)$ is defined as the quotient of $\kk \widetilde{B}_0$ by the same Hecke relations as in the definition of $H_0$. If the short exact sequence~\eqref{eq:ses} splits, then $\kk \widetilde{B}_0$ is a semidirect product of $\kk B_0$ with the group $N_W(W_0)/W_0$, and consequently $\widetilde{H}_0$ is a semidirect product of $H_0$ with the group $N_W(W_0)/W_0$.

As a general notational convention, on those occasions when we need to consider reflection subgroups of $W$ other than our fixed $W_0$, we denote them as $W_1$ or $W_2$, etc. Our notation for the objects associated to $W_i$ is then obtained by replacing $0$ by $i$ in the notation for the analogous objects for $W_0$.


\section{Reflection subgroups of Coxeter groups}
\label{sec:coxeter}

Our aim in this section is to prove that the short exact sequence~\eqref{eq:ses} does split in the case when $W$ is a finite Coxeter group, that is, the complexification of a real reflection group, and $W_0$ is an arbitrary reflection subgroup. Our argument applies equally well to infinite Coxeter groups, as long as we make the assumption stated as~\eqref{eq:assumption} below. For most of this section, we let $(W,S)$ be an arbitrary Coxeter system with $W$ and $S$ possibly infinite. From Section~\ref{ss:topology} onwards we re-impose the assumption that $W$ is finite (except for Remarks~\ref{rem:artin-hom} and~\ref{rem:parabolic-injectivity}).

In the setting of an arbitrary Coxeter system, we use the letter $B$ to denote the Artin group of $(W,S)$ (which is consistent with our previous usage if $W$ is finite; see Section~\ref{ss:topology}). Let $\Sigma$ be the standard set of generators of $B$ which is in canonical bijection with the (possibly infinite) set $S$ of generators of $W$. As a notational convention, if $s$ or $s_i$ or $s_{i_1}$, for example, denotes an element of $S$, we write the corresponding element of $\Sigma$ as $\sigma$ or $\sigma_i$ or $\sigma_{i_1}$, respectively. By definition of $B$, we have a projection homomorphism $\pi:B\onto W$ which extends and is uniquely determined by the bijection from $\Sigma$ to $S$. Recall that the pure Artin group $P:=\ker(B\onto W)$ is generated by elements of the form $\beta \sigma^2\beta^{-1}$ where $\beta\in B$ and $\sigma\in\Sigma$. The projection $\pi$ has a (non-homomorphic) section $W\to B:w\mapsto\underline{w}$ where $\underline{w}$ is the positive lift of $w$: explicitly, if $w=s_1s_2\cdots s_k$ is a reduced expression, then $\underline{w}=\sigma_1\sigma_2\cdots\sigma_k$.

\subsection{Reflection subgroups and normalizers}

We refer the reader to \cite{DYER_SUBGROUPS} for those results stated in this subsection for which no specific reference is given. 

We denote by $T$ the set $\bigcup_{w\in W} w S w^{-1}$ of reflections of $W$. We define the (left) inversion set of $w\in W$ as $$N(w):=\{ t\in T; \ell(tw)<\ell(w)\},$$ where $\ell$ is the usual length function relative to the simple system $S$.
Given any reduced expression $s_1 s_2\cdots s_k$ for $w\in W$, we have $$N(w)=\{ s_1, s_1 s_2 s_1, \dots, s_1 s_2\cdots s_{k-1} s_k s_{k-1} \cdots s_2 s_1\},$$ where the elements listed on the right-hand side are distinct. In particular, we have $|N(w)|=\ell(w)$. 

We have defined $N(w)$ in terms of left inversions, following~\cite{DYER_SUBGROUPS}, which means that the right inversion set of $w$ is
\begin{equation} \label{eq:right-inversions}
N(w^{-1})=\{ t\in T; \ell(wt)<\ell(w)\}=w^{-1}N(w)w.
\end{equation}
For any $w_1,w_2\in W$ we have the cocycle rule  
\begin{equation} \label{eq:cocycle}
N(w_1w_2)=N(w_1)+w_1N(w_2)w_1^{-1},
\end{equation} 
where on the right-hand side $+$ means symmetric difference. So $\ell(w_1 w_2)=\ell(w_1)+\ell(w_2)$ if and only if $N(w_1)\cap w_1N(w_2)w_1^{-1}=\emptyset$, which is equivalent to $N(w_1^{-1})\cap N(w_2)=\emptyset$.

Let $W_0$ be a reflection subgroup of $W$, that is, a subgroup generated by a (possibly infinite) subset of $T$.
\begin{lemma}\label{lem:dyer} \textup{(A special case of \cite[Theorem 3.3]{DYER_SUBGROUPS}.)} 
The reflection subgroup $W_0$ is a Coxeter group in a canonical way, with \textup{(}possibly infinite\textup{)} Coxeter generating set given by $$S_0=\{t\in T; N(t)\cap W_0=\{t\}\}.$$ Relative to this Coxeter structure, the set of reflections of $W_0$ is $T\cap W_0$
and the inversion set of $w\in W_0$ is $N(w)\cap W_0$.  
\end{lemma}

\begin{lemma}\label{lem:dyer2} \textup{(See~\cite[Corollary 3.4(ii)]{DYER_SUBGROUPS}.)} 
For $w\in W$, the following conditions are equivalent:
\begin{enumerate}
\item $w$ has minimal length in its coset $W_0 w$;
\item $N(w)\cap W_0=\emptyset$.
\end{enumerate}
Moreover, in any coset $W_0 x\subseteq W$ there is a unique element which satisfies these conditions.
\end{lemma}

\begin{proof}
It is clear that (1) implies (2), and that in any coset $W_0 x$ there is at least one element satisfying (1). From Lemma~\ref{lem:dyer} we know that the identity is the only element of $W_0$ satisfying (2), and using~\eqref{eq:cocycle} it is easy to deduce that in any coset $W_0 x$ there is at most one element satisfying (2). The result follows.
\end{proof}

We are interested in the normalizer $N_W(W_0)$ of $W_0$ in $W$. Define $$U_0:=\{w\in N_W(W_0); N(w)\cap W_0=\emptyset\}.$$ From~\eqref{eq:right-inversions} we see that this definition would be unchanged if we used right inversions:
$$U_0=\{w\in N_W(W_0); N(w^{-1})\cap W_0=\emptyset\}.$$

\begin{lemma} \label{lem:semi-direct} \textup{(See~\cite[Lemma 2 and Corollary 3]{HOWLETT}.)}
$U_{0}$ is a subgroup of $N_W(W_0)$ which is complementary to $W_0$. Thus we have a semidirect product decomposition $$N_W(W_0)= W_0 \rtimes U_{0}.$$ The conjugation action of $U_0$ on $W_0$ preserves the Coxeter generating set $S_0$.
\end{lemma}

\begin{proof}
That $U_0$ is a subgroup of $N_W(W_0)$ follows easily from~\eqref{eq:cocycle}, and Lemma~\ref{lem:dyer2} implies that it is complementary to $W_0$. Finally, if $u\in U_0$ and $t\in S_0$, then 
\begin{equation*}
\begin{split}
N(utu^{-1})\cap W_0&=(N(u)+uN(t)u^{-1}+utN(u^{-1})tu^{-1})\cap W_0\\
&=(N(u)\cap W_0)+u(N(t)\cap W_0)u^{-1}+ut(N(u^{-1})\cap W_0)tu^{-1}\\
&=\emptyset+\{utu^{-1}\}+\emptyset=\{utu^{-1}\},
\end{split}
\end{equation*}
so $utu^{-1}\in S_0$ as required.
\end{proof}

\begin{example} \label{ex:easy}
Let $(W,S)$ be of type $G_2$, with $S:=\{s_1,s_2\}$. That is, $$W=\langle s_1,s_2\;|\;s_1^2=s_2^2=1, s_1s_2s_1s_2s_1s_2=s_2s_1s_2s_1s_2s_1\rangle.$$ Let $W_0<W$ be the subgroup of type $A_2$ generated by the reflections $s_1$ and $s_2s_1s_2$. Then it is easy to see that $S_0=\{s_1,s_2s_1s_2\}$ and $U_0=\{1,s_2\}$. In this case $W_0$ is normal in $W$, and the semidirect product decomposition of Lemma~\ref{lem:semi-direct} is $W=W_0\rtimes U_0$. We will return to this simple example later in the section.
\end{example}

\begin{remark} \label{rem:roots}
An alternative interpretation of the complementary subgroup $U_0$ is in terms of roots. We will not use this point of view in any proofs, but we describe it briefly for use in examples and remarks. Form a geometric representation $V$ of $(W,S)$ as in~\cite[Section 4]{DYER_SUBGROUPS}, and let $\Pi=\{\alpha_s\,;\, s\in S\}$ denote the given basis of $V$ in canonical bijection with $S$. Let $\Phi=W\Pi$ be the set of all roots, and $\Phi^+=\{\alpha_t\,;\, t\in T\}$ the set of positive roots in canonical bijection with $T$. Then $\Phi$ is the disjoint union of $\Phi^+$ and $\Phi^-:=-\Phi^+$, and for $w\in W$ we have $$N(w)=\{t\in T\,;\, w^{-1}(\alpha_t)\in\Phi^-\}$$ as in~\cite[Lemma 4.3]{DYER_SUBGROUPS}. 

The reflection subgroup $W_0$ gives rise to the subsets $\Phi_0^+:=\{\alpha_t\,;\,t\in T\cap W_0\}$ and $\Pi_0:=\{\alpha_t\,;\,t\in S_0\}$ of $\Phi^+$ and the subset $\Phi_0^-:=-\Phi_0^+$ of $\Phi^-$. As a consequence of~\cite[Theorem 4.4]{DYER_SUBGROUPS}, $\Phi_0:=W_0\Pi_0$ is the disjoint union of $\Phi_0^+$ and $\Phi_0^-$, and every element of $\Phi_0^+$ can be written as a positive linear combination of elements of $\Pi_0$. 

For any $w\in W$, we have $N(w^{-1})\cap W_0=\emptyset$ if and only if $w(\Phi_0^+)\subseteq\Phi^+$, which by the last remark is equivalent to $w(\Pi_0)\subseteq\Phi^+$. Combining this with the fact that the conjugation action of $U_0$ preserves $S_0$, we see that $$U_0=\{w\in W\,;\,w(\Phi_0^+)=\Phi_0^+\}=\{w\in W\,;\,w(\Pi_0)=\Pi_0\}.$$
\end{remark}

\begin{example}\label{ex:d_4}
Let $(W,S)$ be of type $D_4$, with $S:=\{s_1, s_2, s_3, s_4\}$, where $s_1,s_2,s_4$ are the simple reflections which commute with each other. The corresponding roots $\alpha_1,\alpha_2,\alpha_4\in\Pi$ are perpendicular to each other for the unique $W$-invariant inner product on $V$, and they are also perpendicular to the highest root $\alpha_1 + \alpha_2 +2\alpha_3+\alpha_4=\alpha_t$ where $t=s_3s_1s_2s_4s_3s_4s_2s_1s_3$. Let $W_0<W$ be the subgroup of type $4A_1$ generated by $s_1,s_2,s_4,t$. Then $S_0=\{s_1, s_2, s_4, t\}$ and $\Pi_0=\{\alpha_1,\alpha_2,\alpha_4,\alpha_t\}$. In this case $U_0=\{w\in W\,;\,w(\Pi_0)=\Pi_0\}=\{ 1, s_3s_1s_2s_3, s_3s_2s_4s_3, s_3s_1s_4s_3\}$ is isomorphic to the Klein 4-group, with its three non-identity elements acting on $\Pi_0$ as the three fixed-point-free involutions.
\end{example}

\begin{remark} \label{rem:conjugacy}
If $W_1$ is another reflection subgroup of $W$ which is conjugate to $W_0$, then by Lemma~\ref{lem:dyer2} we can find $\tilde{w}\in W$ such that $W_0=\tilde{w}W_1\tilde{w}^{-1}$ and $N(\tilde{w})\cap W_0=\emptyset$. A calculation which is very similar to that in the proof of Lemma~\ref{lem:semi-direct} shows that $S_0=\tilde{w}S_1\tilde{w}^{-1}$ and $U_0=\tilde{w}U_1\tilde{w}^{-1}$. This observation is particularly useful when $W_0$ is a parabolic subgroup of $W$, i.e.\  a conjugate of a standard parabolic subgroup $W_1=\langle S_1\rangle $ for some subset $S_1\subseteq S$. (Our use of the notation $S_1$ is consistent, because it does equal the canonical Coxeter generating set of $W_1$.) The complementary subgroup $U_1$ for such a standard parabolic subgroup $W_1 \leq W$ was described by Howlett~\cite{HOWLETT} in the case when $W$ is finite and by Brink--Howlett~\cite{BRINKHOWLETT} in general.
\end{remark}

\subsection{Reducing non-reduced expressions}

The well-known Deletion Condition states that if $(s_1,\dots,s_k)$ is a sequence of elements of $S$ such that $s_1\cdots s_k$ is non-reduced, meaning that $\ell(w)<k$ where $w=s_1\cdots s_k$, then there exist $a_1,b_1\in\{1,\dots,k\}$ with $a_1<b_1$ such that $$w=s_1\cdots\widehat{s_{a_1}}\cdots\widehat{s_{b_1}}\cdots s_k.$$ Moreover, one can in fact stipulate that $a_1\in\{1,\dots,k\}$ is maximal such that $s_{a_1}\cdots s_k$ is non-reduced; we make this choice of $a_1$ henceforth, and it determines a unique choice of $b_1$. (In fact, $b_1$ is the smallest integer greater than $a_1$ such that $s_{a_1}\cdots s_{b_1}$ is non-reduced.) Note that $s_{a_1+1}\cdots s_{b_1}$ and $s_{a_1}\cdots s_{b_1-1}$ are reduced expressions for the same element. Moreover, $s_{a_1+1}\cdots \widehat{s_{b_1}}\cdots s_k$ is reduced, because $$\ell(s_{a_1+1}\cdots \widehat{s_{b_1}}\cdots s_k)=\ell(s_{a_1}\cdots s_k)=k-a_1-1$$ since $s_{a_1+1}\cdots s_k$ is reduced by the choice of $a_1$.

Now if $s_1\cdots\widehat{s_{a_1}}\cdots\widehat{s_{b_1}}\cdots s_k$ is still not reduced, we can define $a_2,b_2\in\{1,\dots,k-2\}$ in the same way, and we have $a_2<a_1$ by the last remark. 

In this way, starting with a sequence $(s_1,\dots,s_k)$ of elements of $S$, we define a sequence of pairs 
\[ (a_1,b_1),(a_2,b_2),\dots,(a_r,b_r)\ \text{with}\ a_1>a_2>\dots>a_r. \] 
After making all the successive deletions of two terms of the sequence indicated by these pairs, one obtains a reduced expression for $w=s_1\cdots s_k$, so that $\ell(w)=k-2r$. (If the original expression $s_1\cdots s_k$ is already reduced, then $r=0$ and the sequence of pairs is empty.)

The relevance of this construction for us lies in the following computation in $B$. 

\begin{lemma} \label{lem:braid-decomposition}
Let $(s_1,\dots,s_k)$ be any sequence of elements of $S$, and define $(a_1,b_1),\cdots,(a_r,b_r)$ as above. If $\sigma_i$ denotes the generator of $B$ corresponding to $s_i$, and $\underline{w}\in B$ denotes the positive lift of $w=s_1\cdots s_k$, then we have the following equality in $B$:
\[
\begin{split}
\sigma_1\cdots\sigma_k
&=(\sigma_1\cdots\sigma_{a_1-1}\sigma_{a_1}^2\sigma_{a_1-1}^{-1}\cdots \sigma_1^{-1})(\sigma_1\cdots\sigma_{a_2-1}\sigma_{a_2}^2\sigma_{a_2-1}^{-1}\cdots \sigma_1^{-1})\\
&\qquad\qquad\cdots(\sigma_1\cdots\sigma_{a_r-1}\sigma_{a_r}^2\sigma_{a_r-1}^{-1}\cdots \sigma_1^{-1})\,\underline{w}.
\end{split}
\]
\end{lemma}

\begin{proof}
Arguing by induction on $r$, we need only prove that if $r\ge 1$, then
$$
\sigma_1\cdots\sigma_k=(\sigma_1\cdots\sigma_{a_1-1}\sigma_{a_1}^2\sigma_{a_1-1}^{-1}\cdots \sigma_1^{-1})
\sigma_1\cdots\widehat{\sigma_{a_1}}\cdots\widehat{\sigma_{b_1}}\cdots \sigma_k.
$$
This follows immediately from the equality $\sigma_{a_1+1}\cdots \sigma_{b_1}=\sigma_{a_1}\cdots \sigma_{b_1-1}$, which holds since $s_{a_1+1}\cdots s_{b_1}=s_{a_1}\cdots s_{b_1-1}$ are reduced expressions of the same element of $W$. 
\end{proof}

\begin{lemma} \label{lem:concatenation}
Let $(s_1,\dots,s_k)$ be a sequence of elements of $S$ obtained by concatenating reduced expressions for two elements of $W$, namely $w_1=s_1\cdots s_p$ and $w_2=s_{p+1}\cdots s_k$. Define $(a_1,b_1),\dots,(a_r,b_r)$ as above. Then $$N(w_1)\cap w_1N(w_2)w_1^{-1}=\{s_1s_2\cdots s_{a_1-1}s_{a_1}s_{a_1-1}\cdots s_2s_1, \dots, s_1s_2\cdots s_{a_r-1}s_{a_r}s_{a_r-1}\cdots s_2s_1\},$$ where the elements listed on the right-hand side are distinct.
\end{lemma}

\begin{proof}
Let $w=w_1w_2=s_1\cdots s_k$. In this proof, to save space, we use the temporary notation
\[
t_{[a,b]}:=s_as_{a+1}\cdots s_{b-1}s_bs_{b-1}\cdots s_{a+1}s_a,\ \text{for}\ 1\leq a<b\leq k.
\]
Recall that $N(w_1)=\{t_{[1,1]},t_{[1,2]},\dots,t_{[1,p]}\}$ and that these $p$ elements are distinct. Note that we have $a_1\le p$, since $s_{p+1}\cdots s_k$ is reduced. So the $r$ elements $t_{[1,a_i]}$ listed in the statement are distinct and all belong to $N(w_1)$. From~\eqref{eq:cocycle} we see that $N(w_1)\cap w_1N(w_2)w_1^{-1}=N(w_1)\setminus N(w)$ and that this set has cardinality $\frac{\ell(w_1)+\ell(w_2)-\ell(w)}{2}=r$, so it suffices to prove that $t_{[1,a_i]}\notin N(w)$ for all $i$.

We prove this last statement by induction on $r$, the $r=0$ case being vacuously true. Assume that $r\geq 1$. From~\eqref{eq:cocycle} we have
\begin{equation}\label{eq:crucial}
N(w)=\{t_{[1,1]},t_{[1,2]},\dots,t_{[1,a_r]}\}+s_1\cdots s_{a_r}N(s_{a_r+1}\cdots s_k)s_{a_r}\cdots s_{a_1}.
\end{equation}
The induction hypothesis applies to the sequence $(s_{a_r+1},\dots,s_k)$, for which the corresponding sequence of pairs is $(a_1-a_r,b_1-a_r),(a_2-a_r,b_2-a_r),\dots,(a_{r-1}-a_r,b_{r-1}-a_r)$, and tells us that $t_{[a_r+1,a_i]}\notin N(s_{a_r+1}\cdots s_k)$ for all $i<r$. As $t_{[1,1]},t_{[1,2]},\dots,t_{[1,p]}$ are all distinct, we conclude that for $i<r$, $t_{[1,a_i]}$ does not belong to either set on the right-hand side of~\eqref{eq:crucial}. On the other hand, by definition of $a_r$ we have $s_{a_r}\in N(s_{a_r+1}\cdots s_k)$, so $t_{[1,a_r]}$ belongs to both sets on the right-hand side of~\eqref{eq:crucial}. In either case we are done.
\end{proof}

\subsection{Properties of positive lifts}
Now we return to considering the constructions associated with the choice of a reflection subgroup $W_0$ of our Coxeter group $W$.
 
We let $K_0$ be the subgroup of $P=\ker(\pi:B\onto W)$ generated by the elements of the form $\beta\sigma^2\beta^{-1}$ where $\beta\in B$ and $\sigma\in\Sigma$ are such that the reflection $\pi(\beta\sigma\beta^{-1})\in T$ does not belong to $W_0$. This is consistent with our previous definition of $K_0$ in the case that $W$ is finite; see Section~\ref{ss:topology} below. It is clear that $K_0$ is normal in $\pi^{-1}(W_0)$ and in $\widehat{B}_0=\pi^{-1}(N_W(W_0))$. Thus, we can still define $\widetilde{B}_0=\widehat{B}_0/K_0$ in this context, along with the projection $\widetilde{\pi}_0:\widetilde{B}_0\onto N_W(W_0)$ induced by $\pi$.

Define a map $$\psi:N_W(W_0)\to\widetilde{B}_0$$ by $\psi(w)=\underline{w} K_0$. Note that $\psi$ is injective, because $\widetilde{\pi}_0\circ\psi=\mathrm{id}$. The map $\psi$ is not a homomorphism but has the following `partial homomorphism' property:
\begin{proposition} \label{prop:lifting-hom}
Let $w_1,w_2\in N_W(W_0)$. If we have $N(w_1^{-1})\cap N(w_2)\cap W_0=\emptyset$, then $\psi(w_1w_2)=\psi(w_1)\psi(w_2)$.
\end{proposition}
\begin{proof}
Let $w=w_1w_2$. We must prove that $\underline{w} K_0=\underline{w_1}\,\underline{w_2} K_0$.
We form a sequence $(s_1,\dots,s_k)$ of elements of $S$ by concatenating a reduced expression for $w_1$ and a reduced expression for $w_2$, as in Lemma~\ref{lem:concatenation}. Apply Lemma~\ref{lem:braid-decomposition} to this sequence: the left-hand side is exactly $\underline{w_1}\,\underline{w_2}$, so it suffices to prove that each of the elements $\sigma_1\cdots\sigma_{a_i-1}\sigma_{a_i}^2\sigma_{a_i-1}^{-1}\cdots\sigma_1^{-1}$ on the right-hand side belongs to $K_0$. By definition of $K_0$, it suffices to show that the reflection $s_1\cdots s_{a_i}\cdots s_1$ does not belong to $W_0$. But by Lemma~\ref{lem:concatenation}, this reflection belongs to $N(w_1)\cap w_1N(w_2)w_1^{-1}$, and our hypothesis is equivalent to $N(w_1)\cap w_1N(w_2)w_1^{-1}\cap W_0=\emptyset$.
\end{proof}

\begin{corollary} \label{cor:artin-lifting}
The restriction $\psi:W_0\into\pi^{-1}(W_0)/K_0$ is multiplicative on reduced expressions for the Coxeter system $(W_0,S_0)$. In other words, we have the following:
\begin{enumerate}
\item If $w_1,w_2\in W_0$ are such that $\ell_0(w_1w_2)=\ell_0(w_1)+\ell_0(w_2)$ where $\ell_0$ denotes the length function on $W_0$ relative to the generating set $S_0$, then $\psi(w_1 w_2)=\psi(w_1)\psi(w_2)$.
\item If $w\in W_0$ has a reduced expression $w=t_1t_2\cdots t_k$ in terms of the generating set $S_0$, then $\psi(w)=\psi(t_1)\psi(t_2)\cdots\psi(t_k)$.
\item With $B_0$ denoting the Artin group of the Coxeter system $(W_0,S_0)$, there is a unique group homomorphism $\widetilde{\psi}:B_0\to\pi^{-1}(W_0)/K_0$ such that, for any $w\in W_0$ with positive lift $\beta\in B_0$, we have $\widetilde{\psi}(\beta)=\psi(w)$.
\end{enumerate} 
\end{corollary}
\begin{proof}
We first prove (1). Since the inversion sets of $w_1^{-1}$ and $w_2$ relative to the Coxeter system $(W_0,S_0)$ are $N(w_1^{-1})\cap W_0$ and $N(w_2)\cap W_0$ respectively, the assumption that $\ell_0(w_1w_2)=\ell_0(w_1)+\ell_0(w_2)$ means that $$(N(w_1^{-1})\cap W_0)\cap (N(w_2)\cap W_0)=\emptyset.$$ Hence the hypothesis of Proposition~\ref{prop:lifting-hom} is satisfied and (1) follows. Now (2) is an immediate consequence of (1), and (3) follows from (2) because the braid relations defining $B_0$ are equalities between positive lifts of reduced expressions. 
\end{proof}

\begin{corollary} \label{cor:u-lifting}
Let $w_1,w_2\in N_W(W_0)$. If at least one of $w_1,w_2$ belongs to $U_0$, then $\psi(w_1w_2)=\psi(w_1)\psi(w_2)$. In particular, we have:
\begin{enumerate}
\item The restriction $\psi:U_0\into\widetilde{B}_0$ is a group homomorphism.
\item For all $u\in U_0$ and $t\in S_0$ we have $\psi(u)\psi(t)=\psi(utu^{-1})\psi(u)$.
\end{enumerate}
\end{corollary}
\begin{proof}
Since $U_0=\{w\in N_W(W_0); N(w)\cap W_0=\emptyset\}=\{w\in N_W(W_0); N(w^{-1})\cap W_0=\emptyset\}$, this follows immediately from Proposition~\ref{prop:lifting-hom}.
\end{proof}

To recover our short exact sequence~\eqref{eq:ses} we need to make the following assumption:
\begin{equation}
\label{eq:assumption}
\text{The homomorphism $\widetilde{\psi}:B_0\to\pi^{-1}(W_0)/K_0$ of Corollary~\ref{cor:artin-lifting}(3) is an isomorphism.}
\end{equation} 
We will show in Proposition~\ref{prop:topology} that~\eqref{eq:assumption} holds when $W$ is finite, using the interpretation of $B$ and $B_0$ as fundamental groups of orbit spaces of hyperplane complements as in Section~\ref{sec:definitions}. Remarks~\ref{rem:artin-hom} and~\ref{rem:parabolic-injectivity} discuss the case when $W$ is infinite.

Under the assumption~\eqref{eq:assumption}, we have an injective homomorphism $B_0\into\widetilde{B}_0$ which is the composition of $\widetilde{\psi}$ with the inclusion of $\pi^{-1}(W_0)/K_0$ in $\widetilde{B}_0$. We will show in Proposition~\ref{prop:topology} that, when $W$ is finite, this injective homomorphism coincides with the one defined previously, so we can without ambiguity identify $B_0$ with a subgroup of $\widetilde{B}_0$ in the current more general setting. It is clear from the definitions that this inclusion of $B_0$ in $\widetilde{B}_0$ again fits into a short exact sequence~\eqref{eq:ses} forming the top row of the commutative diagram~\eqref{eq:comm-diag} where the bottom row is the short exact sequence~\eqref{eq:ses-triv}. We can now state the main result of this section.

\begin{theorem} \label{thm:splitfiniteCox}
Let $W$ be a Coxeter group and $W_0$ a reflection subgroup of $W$ such that assumption~\eqref{eq:assumption} holds; for example, this holds if $W$ is finite. Then the splitting of the short exact sequence~\eqref{eq:ses-triv} given by Lemma~\ref{lem:semi-direct} lifts to a splitting of the short exact sequence~\eqref{eq:ses}. Namely, after identifying $N_W(W_0)/W_0$ with $U_0$, the splitting of~\eqref{eq:ses} is the homomorphism $\psi:U_0\into\widetilde{B}_0$. Hence we have a semidirect product decomposition $$\widetilde{B}_0=B_0\rtimes\psi(U_0).$$ The conjugation action of $\psi(U_0)$ on $B_0$ preserves the Artin generating set $\Sigma_0$ of $B_0$, and for $u\in U_0$, the action of $\psi(u)$ on $\Sigma_0$ is the same as the conjugation action of $u$ on $S_0$.
\end{theorem}

\begin{proof}
This follows from Corollary~\ref{cor:u-lifting} in view of~\eqref{eq:assumption}.
\end{proof}

\begin{example} \label{ex:easy-ctd}
We return to the setting of Example~\ref{ex:easy}, with $W=\langle s_1,s_2\rangle$ of type $G_2$ and $W_0=\langle s_1,s_2s_1s_2\rangle$ of type $A_2$. In this case, the Artin group corresponding to $W$ is $$B=\langle\sigma_1,\sigma_2\;|\;\sigma_1\sigma_2\sigma_1\sigma_2\sigma_1\sigma_2=\sigma_2\sigma_1\sigma_2\sigma_1\sigma_2\sigma_1\rangle.$$ The relevant subgroups of $B$ are as follows: 
\begin{itemize}
\item $K_0$ is the subgroup of $B$ normally generated by $\sigma_2^2$, 
\item $\pi^{-1}(W_0)$ is the subgroup of $B$ normally generated by $\sigma_1$ and $\sigma_2^2$, 
\item $\widehat{B}_0=B$ itself, since $W_0$ is normal in $W$.
\end{itemize} 
Clearly $\psi:W\to\widetilde{B}_0=B/K_0$ is not a homomorphism, because $\psi(s_1)^2=\sigma_1^2K_0\neq 1K_0=\psi(s_1^2)$. However, as an example of Corollary~\ref{cor:artin-lifting}, when $w_1=s_2s_1s_2s_1$ and $w_2=s_2s_1s_2$, we have 
$$\psi(w_1)\psi(w_2)=\sigma_2\sigma_1\sigma_2\sigma_1\sigma_2\sigma_1\sigma_2K_0=\sigma_2^2\sigma_1\sigma_2\sigma_1\sigma_2\sigma_1K_0=
\sigma_1\sigma_2\sigma_1\sigma_2\sigma_1K_0=\psi(w_1w_2).$$
Moreover, the restriction of $\psi$ to $U_0=\langle s_2\rangle$ is a homomorphism, in accordance with Corollary~\ref{cor:u-lifting}.
Corollary~\ref{cor:artin-lifting}(3) and the assumption~\eqref{eq:assumption}, true for finite $W$, imply that we can regard the Artin group $B_0$ of type $A_2$ as a subgroup of $\widetilde{B}_0$ by identifying the Artin generators of $B_0$ with $\sigma_1K_0$ and $\sigma_2\sigma_1\sigma_2K_0$. Then Theorem~\ref{thm:splitfiniteCox} states that we have a semidirect product decomposition $\widetilde{B}_0=B_0\rtimes\langle\sigma_2K_0\rangle$, where the conjugation action of $\psi(s_2)=\sigma_2 K_0$ interchanges $\sigma_1 K_0$ and $\sigma_2\sigma_1\sigma_2K_0$.
\end{example}

\begin{example}
Continue the notation of Example~\ref{ex:d_4}, where $W$ is of type $D_4$ and $W_0$ of type $4A_1$. In this case, the Artin group $B_0$ is $\Z^4$, identified with the subgroup of $\widetilde{B}_0=\widehat{B}_0/K_0$ generated by $$\sigma_1K_0, \sigma_2K_0, \sigma_4K_0, \sigma_3\sigma_1\sigma_2\sigma_4\sigma_3\sigma_1\sigma_2\sigma_4\sigma_3 K_0.$$ Theorem~\ref{thm:splitfiniteCox} states that we have a semidirect product decomposition $\widetilde{B}_0=\Z^4\rtimes\psi(U_0)$, where $\psi(U_0)$ is the Klein 4-group $\langle \sigma_3\sigma_1\sigma_2\sigma_3K_0,\sigma_3\sigma_2\sigma_4\sigma_3K_0\rangle$.
\end{example}

\subsection{The finite Coxeter case}
\label{ss:topology}

For the remainder of the section we assume that $W$ is a finite Coxeter group. Thus we are in the setting of Section~\ref{sec:definitions}, but in addition there is a real form $V$ of the complex reflection representation $\C^n$ on which $W$ acts as a finite real reflection group. We endow $V$ with a $W$-invariant inner product. For $t\in T$, we write $H_t$ for the hyperplane $\ker(t-\mathrm{id})$ in $\C^n$, and $V\cap H_t$ for the real hyperplane in $V$. We have a bijection $T\overset{\sim}{\to}\mathcal{A}:t\mapsto H_t$. The real hyperplane complement $V\cap X=V\setminus\bigcup_{t\in T}(V\cap H_t)$ is the union of contractible connected components called chambers, which are permuted simply transitively by $W$. Let $C$ be a chamber which is compatible with the simple system $S$, i.e.\ a chamber whose walls are open subsets of the hyperplanes $V\cap H_s$ for $s\in S$. Then the chambers adjacent to $C$ are those of the form $s(C)$ for $s\in S$. For any $w\in W$, the inversion set $N(w)$ consists exactly of those reflections $t\in T$ such that $V\cap H_t$ separates $C$ from $w(C)$.

We choose our base-point $\tilde{x}$ to belong to $C$. The two groups for which we have used the notation $B$, namely the Artin group of $(W,S)$ and the fundamental group $\pi_1(X/W,[\tilde{x}]_W)$, can be identified in a standard way so that the natural projections $\pi:B\onto W$ coincide. Recall how this standard identification works on the generators: for $s\in S$, the generator $\sigma=\underline{s}$ of the Artin group is identified with a special choice of braided reflection $\sigma_{H_s}$. As in~\cite{BRIESKORN,BMR} (for instance), this braided reflection is defined to be the homotopy class of the image in $X/W$ of a specific path from $\tilde{x}$ to $s(\tilde{x})$ in $X$, which is a perturbation of the straight-line path from $\tilde{x}$ to $s(\tilde{x})$ in $V$. Note that $s(\tilde{x})$ belongs to the chamber $s(C)$ which is adjacent to $C$. Let $\tilde{x}_s\in\overline{C}$ denote the orthogonal projection of $\tilde{x}$ onto the common wall of these chambers. Then the straight-line path from $\tilde{x}$ to $s(\tilde{x})$ in $V$ crosses $V\cap H_s$ at $\tilde{x}_s$, and does not intersect any other hyperplane in $\mathcal{A}$. To produce the desired path from $\tilde{x}$ to $s(\tilde{x})$ in $X$, one removes a small interval centred at $\tilde{x}_s$ from the straight-line path in $V$, and replaces it with a semicircle in $\C^n$ around $H_s$.

The element $\sigma^2\in P$ is then identified with the homotopy class of a special meridian around $H_s$, namely the loop in $X$ which travels on the straight line from $\tilde{x}$ almost as far as $\tilde{x}_s$, then traverses a full circle in $\C^n$ around $H_s$, and then returns along the same straight line to $\tilde{x}$. For any $\beta\in B$, the element $\beta\sigma^2\beta^{-1}\in P$ is a meridian around $H_{\pi(\beta\sigma\beta^{-1})}$, and every meridian around a hyperplane in $\mathcal{A}$ is of this form for some $\beta\in B$ and $\sigma\in\Sigma$. This is why the two descriptions we have given of generating sets for $P=\ker(B\onto W)$, and the two definitions we have given of its subgroup $K_0$, are consistent.     

The reflection subgroup $W_0$ has its own chambers, the connected components of $V\cap X^0$, each of which contains a number of chambers for $W$. Let $C_0$ be the unique chamber for $W_0$ which contains $C$, so in particular $\tilde{x}\in C_0$. Then the walls of $C_0$ are open subsets of the hyperplanes $V\cap H_t$ for $t\in S_0$. Replacing $(W,S)$ with $(W_0,S_0)$ in the above, we get an analogous standard identification between the two groups for which we have used the notation $B_0$, namely the Artin group of $(W_0,S_0)$ and the fundamental group $\pi_1(X^0/W_0,[\tilde{x}]_{W_0})$. 

On the other hand, in Corollary~\ref{cor:artin-lifting} we saw a homomorphism $\widetilde{\psi}$ from the Artin group of $(W_0,S_0)$ to the subquotient $\pi^{-1}(W_0)/K_0$ of the Artin group of $(W,S)$, uniquely specified by the non-homomorphic map $\psi:W_0\into\pi^{-1}(W_0)/K_0$ defined by $\psi(w)=\underline{w}K_0$. Once we make the standard identification of the Artin group of $(W,S)$ with $\pi_1(X/W,[\tilde{x}]_W)$, the subquotient $\pi^{-1}(W_0)/K_0$ becomes identified with $\pi_1(X^0/W_0,[\tilde{x}]_{W_0})$ as we saw in Section~\ref{sec:definitions}. So $\widetilde{\psi}$ becomes a homomorphism from $B_0$ to itself. The following result proves that~\eqref{eq:assumption} holds when $W$ is finite, and also that our use of the notation $B_0$ is consistent and unambiguous.

\begin{figure}
\begin{center}
\begin{tikzpicture}[scale=3]
\fill[gray] (1,0) -- (1,1) -- (0,1) -- (0,0) -- cycle;
\fill[blue] (1,0) -- (1,0.5773502693) -- (0,0) -- cycle;
\draw[thick] (-1,0) -- (1,0);
\draw[thick] (0,-1) -- (0,1);
\draw[blue] (0.8660254040,0.5) -- (-0.8660254040,-0.5); 
\draw[blue] (0.5,0.8660254040) -- (-0.5,-0.8660254040);
\draw[blue] (-0.8660254040,0.5) -- (0.8660254040,-0.5); 
\draw[blue] (-0.5,0.8660254040) -- (0.5,-0.8660254040);
\node[red] (PB) at (0.9*0.9659258263,0.9*0.2588190451) {$\bullet$};
\node[red] (PBR) at (-0.9*0.9659258263,0.9*0.2588190451) {$\bullet$};
\draw[red] (0.9*0.9659258263,0.9*0.2588190451) -- (0.9*
0.7071067810,0.9*
0.7071067810);
\draw[red] (0.9*
0.7071067810,0.9*
0.7071067810) -- (0.9*0.2588190451,0.9*0.9659258263);
\draw[red] (-0.9*0.2588190451,0.9*0.9659258263) -- (-0.05,0.9*0.9659258263);
\draw[red] (0.05,0.9*0.9659258263) -- (0.9*0.2588190451,0.9*0.9659258263);
\draw[red] (-0.9*0.2588190451,0.9*0.9659258263) -- (-0.9*
0.7071067810,0.9*
0.7071067810);
\draw[red] (-0.9*
0.7071067810,0.9*
0.7071067810) -- (-0.9*0.9659258263,0.9*0.2588190451);
\draw[red] (-0.9*0.9659258263,0.9*0.2588190451) -- (-0.05,0.9*0.2588190451);
\draw[red] (0.05,0.9*0.2588190451) -- (0.9*0.9659258263,0.9*0.2588190451);

\draw (1.2,0) node {$H_{s_1}$};
\draw (0.8660254040+0.3,0.5) node {$H_{s_2}$};
\draw (0.5+0.2,0.8660254040) node {$H_{s_2s_1s_2}$};
\draw (0,1.1) node {$H_{t}$};
\draw[blue] (1.2,0.3) node {$C$};
\draw[gray] (1.2,1) node {$C_0$};
\draw[red] (-0.3,0.3) node {$\mathcal{P}_1$};
\draw[red] (-0.3,0.95) node {$\mathcal{P}_2'$};
\draw[red] (-1.05,0.25) node {$t(\tilde{x})$};
\draw[red] (1.05,0.25) node {$\tilde{x}$};
\draw[red] (0.05,0.9*0.2588190451) arc (0:180:0.05);
\draw[red] (0.05,0.9*0.9659258263) arc (0:180:0.05);

\end{tikzpicture}
\end{center}
\caption{Chambers and paths for $\langle s_1,t \rangle < G_2$ with $t =s_2 s_1 s_2 s_1 s_2$ }
\label{fig:artingens}
\end{figure}

\begin{proposition} \label{prop:topology}
Interpreted as above, $\widetilde{\psi}:B_0\to B_0$ is the identity. Equivalently, for any $w\in W_0$, $\psi(w)$ equals the positive lift of $w$ to $B_0$ relative to the Coxeter generating set $S_0$. 
\end{proposition}

\begin{proof}
In view of Corollary~\ref{cor:artin-lifting}, we need only check that, for any $t\in S_0$, $\psi(t)$ equals the corresponding Artin generator of $B_0$. For this, let $t=s_1s_2\cdots s_k\cdots s_2s_1$ be a palindromic reduced expression for $t$ in the generating set $S$. (Recall that any reflection $t$ has a palindromic reduced expression: starting with an arbitrary reduced expression $s_1 s_2\cdots s_{2k-1}$ for $t$, one can easily show as in~\cite[Lemma 2.7]{DYER_SUBGROUPS} that $t=s_1 s_2 \cdots s_k\cdots s_2 s_1$.)

We have two different ways to define a loop in $X^0/W_0$ based at $[\tilde{x}]_{W_0}$ associated to $t$, and we need to check that they are homotopic. The standard loop $\mathcal{L}_1$, whose homotopy class is the generator of $B_0$ corresponding to $t$, is the image in $X^0/W_0$ of the path $\mathcal{P}_1$ from $\tilde{x}$ to $t(\tilde{x})$ in $X^0$ obtained by perturbing the straight-line path in $V$, replacing a small interval centred on $\tilde{x}_t$ with a semicircle in $\C^n$ around $H_t$. 

The alternative loop $\mathcal{L}_2$, whose homotopy class is $\psi(t)$, actually lies in the subset $X/W_0\subseteq X^0/W_0$, and its image in $X/W$ has homotopy class $\sigma_1\sigma_2\cdots \sigma_k\cdots\sigma_2\sigma_1\in B$. One can construct $\mathcal{L}_2$ as the image in $X/W_0$ of a path $\mathcal{P}_2$ from $\tilde{x}$ to $t(\tilde{x})$ in $X$ obtained by perturbing the piecewise-linear path $\mathcal{P}_3$ in $V$ which travels from $\tilde{x}$ on a straight line to $s_1(\tilde{x})$, thence on a straight line to $s_1s_2(\tilde{x})$, thence on a straight line to $s_1s_2s_3(\tilde{x})$, and so on until one reaches $t(\tilde{x})$. Note that the straight line segments of $\mathcal{P}_3$ cross exactly one hyperplane each, namely the hyperplanes corresponding to the elements of $N(t)$ in the order listed as follows:
\[
N(t)=\{s_1, s_1s_2s_1,\cdots,s_1s_2\cdots s_k\cdots s_2 s_1,\cdots,s_1s_2\cdots s_k\cdots s_2 s_1 s_2 \cdots s_k \cdots s_2s_1\}.
\]
To obtain the path $\mathcal{P}_2$, one perturbs $\mathcal{P}_3$ by replacing small intervals centred on the various hyperplane crossing points with semicircles in $\C^n$ about the hyperplanes.

Since $t\in S_0$, the only element of $N(t)\cap W_0$ is $t$ itself, appearing in the middle position in the above list. Accordingly, of the hyperplanes which $\mathcal{P}_3$ crosses, only the middle one $H_t$ belongs to $\mathcal{A}_0$. Hence, when viewed as a path from $\tilde{x}$ to $t(\tilde{x})$ in $X^0$ rather than $X$, $\mathcal{P}_2$ is homotopic to another path $\mathcal{P}_2'$ which coincides with $\mathcal{P}_3$ except for maintaining the perturbation about the hyperplane $H_t$. Note that the first half of $\mathcal{P}_3$ lies entirely in $C_0$, and the second half lies entirely in $t(C_0)$. Since these chambers are contractible, $\mathcal{P}_2'$ is homotopic to the standard path $\mathcal{P}_1$, and we are done. (See Figure \ref{fig:artingens} for a picture showing the paths $\mathcal{P}_1$ and $\mathcal{P}_2'$ in a sample case.)
\end{proof}

\begin{remark} \label{rem:artin-hom}
The same argument shows that~\eqref{eq:assumption} also holds when $W$ is of affine type. In this case one can interpret $W$ as a reflection group on a Euclidean space $V$ with an arrangement consisting of affine rather than linear hyperplanes, and with chambers of finite volume, as in~\cite[Ch.~V, \S4, no.~9]{BOURBAKI}. The statements relating to $W$ and $B$ in the discussion before Proposition~\ref{prop:topology} remain true in this context, as seen in~\cite{SALVETTI,VANDERLEK} (for instance). By the classification of reflection subgroups of affine $W$ given in~\cite[Theorem 5.1(ii)]{DYER_SUBGROUPS}, the representation of $W_0$ on $V$ is a direct product of standard reflection representations of finite and affine Coxeter groups. Consequently, the statements relating to $W_0$ and $B_0$ in the discussion before Proposition~\ref{prop:topology} also remain true. The proof of Proposition~\ref{prop:topology} then goes through unchanged.  
\end{remark}

\begin{remark} \label{rem:parabolic-injectivity}
We do not yet know for which reflection subgroups of more general Coxeter groups the statement~\eqref{eq:assumption} holds. One observation we can make is that when $S$ is finite and $W_0$ is a parabolic subgroup of $W$, the homomorphism $\widetilde{\psi}$ is injective. To prove this, using Remark~\ref{rem:conjugacy} we can assume that $S_0\subseteq S$. In this case it was shown by van der Lek~\cite[Theorem 4.13]{VANDERLEK} that the natural homomorphism $B_0\to B$, mapping each Artin generator of $B_0$ to the corresponding Artin generator of $B$, is injective, so we can identify $B_0$ with a subgroup of $\pi^{-1}(W_0)$. Then $\widetilde{\psi}$ is the composition of the inclusion $B_0\into \pi^{-1}(W_0)$ with the projection $\pi^{-1}(W_0)\onto \pi^{-1}(W_0)/K_0$, so it is injective if and only if $P_0\cap K_0=\{1\}$. But by the proof of~\cite[Lemma 4.11]{VANDERLEK}, the inclusion $P_0\into P$ has a left inverse $P\onto P_0$, and from the topological definition of the latter it is clear that each generator of $K_0$ belongs to $\ker(P\onto P_0)$.     
\end{remark}

\subsection{Hecke algebras}
\label{ss:std}

Continue to assume that $W$ is a finite Coxeter group. The semidirect product decomposition of $\widetilde{B}_0$ induces a semidirect product decomposition of the corresponding Hecke algebra $\widetilde{H}_0$, which allows us to write down a standard basis for that algebra. Recall that for $w\in W_0$, the image of the positive lift $\psi(w)\in B_0$ in the Hecke algebra $H_0$ is written $T_w$, and the elements $\{T_w\}_{w\in W_0}$ form the standard basis of $H_0$. We simply extend this notation to $w\in N_W(W_0)$, writing $T_w$ for the image in $\widetilde{H}_0$ of $\psi(w)\in\widetilde{B}_0$.

\begin{theorem} \label{thm:std}
The elements $\{T_w\}_{w\in N_W(W_0)}$ form a basis of $\widetilde{H}_0$. The subset $\{T_w\}_{w\in W_0}$ spans a subalgebra which can be identified with $H_0$ with its standard basis. The subset $\{T_u\}_{u\in U_0}$ spans a subalgebra which can be identified with the group algebra $\kk U_0$ with its obvious basis. Multiplication induces a $\kk$-module isomorphism $H_0\otimes_{\kk} \kk U_0 \simto \widetilde{H}_0$ and we have $$T_w T_u = T_{wu}=T_u T_{u^{-1}wu}\ \text{for}\ w\in W_0, u\in U_0.$$   
\end{theorem}

\begin{proof}
This follows by combining the definition of $\widetilde{H}_0$ with Theorem~\ref{thm:splitfiniteCox} and Corollary~\ref{cor:u-lifting}. 
\end{proof}

Let $(S_U, R_U)$ be a presentation by generators and relations for the monoid $U_0$, where $S_U$ is stable under taking inverses. It follows from Theorem~\ref{thm:std} that a presentation of $\widetilde{H}_0$ is obtained by taking as generating set $\{T_s\}_{s\in S_0} \cup \{T_u\}_{u\in S_U}$ with the following relations:
\begin{itemize}
\item the relations of the Hecke algebra $H_0$ on the elements $T_s$ for $s\in S_0$,
\item the relations $R_U$ on the elements $T_u$ for $u\in S_U$ (which entail in particular that $T_uT_{u^{-1}}=T_{u^{-1}}T_u=1$ for all $u\in S_U$),
\item the relations $T_{u^{-1}} T_s T_u=T_{u^{-1} s u}$ for all $u\in S_U, s\in S_0$. 
\end{itemize}
Recall that when $W_0$ is a parabolic subgroup of $W$, a presentation $(S_U,R_U)$ of $U_0$ can be given using the results of~\cite{BRINKHOWLETT, HOWLETT}. We are not aware of any similarly nice recipe for a presentation of $U_0$ when $W_0$ is an arbitrary reflection subgroup of $W$, but see Lemma~\ref{lem:presentation2} below for a related statement.


\section{Groupoid descriptions of normalizers}
\label{sec:groupoids}

In this section we present an alternative proof of the splitting of~\eqref{eq:ses}, which we find enlightening. The main idea is to adopt a more canonical point of view: instead of choosing a reflection subgroup $W_0$ of our Coxeter group $W$, we consider a groupoid (or rather, several groupoids) involving not just $W_0$ but all its conjugate subgroups. Thus we in fact upgrade the statement about groups to one about groupoids. The constructions and results of this section are valid for an arbitrary (not necessarily finite) Coxeter group, except for the splitting of the short exact sequence~\eqref{eq:ses} (proved after Theorem~\ref{thm:splitCox-groupoid} below), which as before is valid under the assumption~\eqref{eq:assumption}.

This idea was inspired by the Brink--Howlett groupoid description~\cite{BRINKHOWLETT} of the subgroup $U_0$ in the case where $W_0$ is a parabolic subgroup of $W$. However, our groupoids are different from theirs, and are defined for reflection subgroups which are not necessarily parabolic. We will comment further on the relationship between our groupoids and theirs in Remark~\ref{rem:brinkhowlett}.

\subsection{Preliminaries on groupoids}

A reference for the small amount of category theory we will need is \cite{MACLANE}. A groupoid is a small category $\mathcal{G}$ in which every morphism is invertible. We say that two objects $x,y$ of $\mathcal{G}$ are in the same connected component if $\Hom_{\mathcal{G}}(x,y)$ is non-empty. If this holds, then the groups $\End_{\mathcal{G}}(x)$ and $\End_{\mathcal{G}}(y)$ are isomorphic, with every morphism $\varphi\in\Hom_{\mathcal{G}}(x,y)$ defining an isomorphism $\End_{\mathcal{G}}(x)\overset{\sim}{\to}\End_{\mathcal{G}}(y)$ by conjugation.

Recall from \cite[Section II.8]{MACLANE} the general concepts of congruences on a category and quotient categories. To specify a congruence on a groupoid amounts to specifying a collection $K_\bullet=(K_x)$ of subgroups $K_x<\End_{\mathcal{G}}(x)$ for each object $x$ of $\mathcal{G}$, satisfying the compatibility condition that for any $\varphi\in\Hom_{\mathcal{G}}(x,y)$ we have $\varphi K_x\varphi^{-1}=K_y$. Note that this condition implies in particular that $K_x\vartriangleleft\End_{\mathcal{G}}(x)$ for every $x$. Given such $K_\bullet=(K_x)$, the quotient groupoid $\mathcal{G}/K_\bullet$ has the same objects as $\mathcal{G}$, and morphism sets $\Hom_{\mathcal{G}/K_\bullet}(x,y) =\Hom_{\mathcal{G}}(x,y)/\sim$, where the equivalence relation $\sim$ is defined by specifying that, for any $\psi,\psi'\in\Hom_{\mathcal{G}}(x,y)$,
$$\psi\sim\psi'\ \Longleftrightarrow\ \psi^{-1}\psi'\in K_x,$$
which is equivalent to $$\psi'\psi^{-1}\in K_y.$$
The composition of morphisms in $\mathcal{G}/K_\bullet$ is induced by that in $\mathcal{G}$. In other words, we have a full functor $\mathcal{G}\onto\mathcal{G}/K_\bullet$ which is the identity on objects and maps each morphism $\varphi\in\Hom_{\mathcal{G}}(x,y)$ to the equivalence class $\varphi K_x=K_y\varphi\in\Hom_{\mathcal{G}/K_\bullet}(x,y)$.

\subsection{Groupoids of reflection subgroups}

\label{ss:groupoids}

Let $(W,S)$ be a Coxeter system. Consider the groupoid $\mathcal{N}$ whose objects are the reflection subgroups of $W$, with 
$$
\Hom_{\mathcal{N}}(W_1,W_2):=\{w\in W; wW_1w^{-1}=W_2\}
$$
and composition given by multiplication in $W$. Thus for any reflection subgroup $W_0$ of $W$, the group $\End_{\mathcal{N}}(W_0)$ is exactly the normalizer $N_W(W_0)$. The connected components of $\mathcal{N}$ are the conjugacy classes of reflection subgroups of $W$. For what follows, it would make no difference if we restricted attention to a single conjugacy class of reflection subgroups, so a reader who prefers groupoids to be connected may imagine that we have done so.

When it is necessary to distinguish between elements of $W$ and the various morphisms in $\mathcal{N}$ which they represent, we will write the elements of $\Hom_{\mathcal{N}}(W_1,W_2)$ as arrows $W_2\overset{w}{\longleftarrow}W_1$. We use left-facing arrows so that the morphisms compose in the expected order:
\[
(W_3\overset{w}{\longleftarrow}W_2\overset{w'}{\longleftarrow}W_1)\ =\ W_3\overset{ww'}{\longleftarrow}W_1.
\]

The main advantage of considering these groupoids is that, although we do not know a general presentation of the group $N_W(W_0)$, it is easy to give a presentation for the groupoid $\mathcal{N}$ as a whole, in the sense of presentations of categories~\cite[Section II.8]{MACLANE}.

\begin{lemma} \label{lem:presentation1}
As a category, $\mathcal{N}$ has the following presentation: 
\begin{itemize}
\item the generators are the morphisms $sW_1s\overset{s}{\longleftarrow} W_1$, for $s$ any element of $S$ and $W_1$ any object of $\mathcal{N}$; 
\item the relations are the Coxeter relations
\[
(W_1 \overset{s}{\longleftarrow} sW_1s \overset{s}{\longleftarrow} W_1)\ =\  \mathrm{id}_{W_1},
\]
\[
\begin{split}
&(\underbrace{stst\cdots}_{m}W_1\underbrace{\cdots tsts}_{m}\overset{s}{\longleftarrow} 
\underbrace{tst\cdots}_{m-1}W_1\underbrace{\cdots tst}_{m-1}\overset{t}{\longleftarrow}
\cdots\longleftarrow W_1)\\
&=
(\underbrace{tsts\cdots}_{m}W_1\underbrace{\cdots stst}_{m}\overset{t}{\longleftarrow} 
\underbrace{sts\cdots}_{m-1}W_1\underbrace{\cdots sts}_{m-1}\overset{s}{\longleftarrow}
\cdots\longleftarrow W_1),
\end{split}
\]
for $s\neq t\in S$ such that $st$ has finite order $m$ in $W$.
\end{itemize}
\end{lemma}
\begin{proof}
This is obvious from the fact that $W$ itself has such a Coxeter presentation.
\end{proof}

Now let $\mathcal{U}$ be the sub-groupoid of $\mathcal{N}$ which has the same set of objects but with
\[
\Hom_{\mathcal{U}}(W_1,W_2):=\{w\in W; wW_1w^{-1}=W_2, N(w)\cap W_2=\emptyset\}.
\] 
Note that by~\eqref{eq:right-inversions} the condition $N(w)\cap W_2=\emptyset$ could be replaced by $N(w^{-1})\cap W_1=\emptyset$. The fact that this condition does define a sub-groupoid of $\mathcal{N}$ is an easy consequence of~\eqref{eq:cocycle}. For any object $W_0$, the group $\End_{\mathcal{U}}(W_0)$ is exactly the complementary subgroup $U_0$ of $W_0$ in $N_W(W_0)$ considered in the previous section.

Let $T_\bullet$ denote the `tautological' collection of subgroups $W_0<\End_{\mathcal{N}}(W_0)$ for all objects $W_0$ of $\mathcal{N}$, and let $\overline{\mathcal{N}}:=\mathcal{N}/T_\bullet$ be the quotient groupoid with $\End_{\overline{\mathcal{N}}}(W_0)=N_W(W_0)/W_0$. The splitting of~\eqref{eq:ses-triv} has the following groupoid version:
\begin{lemma} \label{lem:semi-direct-groupoid}
The composition $\mathcal{U}\into\mathcal{N}\onto\overline{\mathcal{N}}$ is an isomorphism of groupoids.
\end{lemma}
\begin{proof}
Note that all the functors involved are the identity on the set of objects. It follows from Lemma~\ref{lem:dyer2} that the connected components of $\mathcal{U}$ are the same as those of $\mathcal{N}$, and hence the same as those of $\overline{\mathcal{N}}$. Therefore the claim follows from the group isomorphism $U_0\overset{\sim}{\to}N_W(W_0)/W_0$ proved in Lemma~\ref{lem:semi-direct}. 
\end{proof}
  
Of the generating morphisms $sW_1s\overset{s}{\longleftarrow} W_1$ of $\mathcal{N}$, those which belong to the sub-groupoid $\mathcal{U}$ are those where $s\notin W_1$, or equivalently $s\notin sW_1s$. (Note that it is possible to have $s\notin W_1$ and $sW_1s=W_1$.) A crucial observation is that these morphisms generate $\mathcal{U}$.

\begin{lemma} \label{lem:presentation2}
As a category, $\mathcal{U}$ has the following presentation: 
\begin{itemize}
\item the generators are those generators $sW_1s\overset{s}{\longleftarrow} W_1$ of $\mathcal{N}$ which belong to $\mathcal{U}$, i.e.\ satisfy the additional condition that $s\notin W_1$; 
\item the relations are the same Coxeter relations as in the above presentation of $\mathcal{N}$, whenever those relations involve only generators belonging to $\mathcal{U}$.
\end{itemize}
\end{lemma}
\begin{proof}
Suppose that $W_2\overset{w}{\longleftarrow}W_1$ is a morphism of $\mathcal{N}$ and let $w=s_1s_2\cdots s_k$ be a reduced expression for $w$. Then $W_2\overset{w}{\longleftarrow}W_1$ equals the following composition of generators of $\mathcal{N}$:
\[
W_2\overset{s_1}{\longleftarrow}s_1W_2s_1\overset{s_2}{\longleftarrow}s_2s_1W_2s_1s_2\overset{s_3}{\longleftarrow}\cdots
\overset{s_k}{\longleftarrow}W_1.
\]
Since $N(w)=\{s_1 s_2\cdots s_{i-1} s_i s_{i-1} \cdots s_2 s_1;1\leq i\leq k\}$, we have
\[
N(w)\cap W_2=\emptyset\ \Longleftrightarrow\ s_i\notin s_{i-1}\cdots s_1W_2s_1\cdots s_{i-1},\ 1\leq i\leq k.
\]  
Thus $W_2\overset{w}{\longleftarrow}W_1$ is a morphism of $\mathcal{U}$ if and only if all the generators involved in the above expression belong to $\mathcal{U}$. The claim now follows from Lemma~\ref{lem:presentation1}. 
\end{proof}

\begin{example} \label{ex:d_4-groupoid}
Let $(W,S)$ be of type $D_4$ and define $W_0=\langle s_1, s_2, s_4, t\rangle$ as in Example~\ref{ex:d_4}, with $t=s_3s_1s_2s_4s_3s_4s_2s_1s_3$. The other reflection subgroups in the conjugacy class of $W_0$ are
\[
\begin{split}
W_1&=s_3W_0s_3=\langle s_1s_3s_1,s_2s_3s_2,s_4s_3s_4,s_1s_2s_4s_3s_4s_2s_1\rangle\ \text{and}\\
W_2&=s_1W_1s_1=s_2W_1s_2=s_4W_1s_4=\langle s_3, s_1s_2s_3s_2s_1, s_1s_4s_3s_4s_1, s_2s_4s_3s_4s_2\rangle.
\end{split}
\]

As a consequence of Lemma~\ref{lem:presentation2}, the connected component of $\mathcal{U}$ with objects $W_0,W_1,W_2$ is completely encoded by the multi-graph where the (bi-directional) edges represent the conjugations by elements of $S$ not belonging to the subgroups involved:
\begin{center}
\begin{tikzpicture}
\draw [<->] (2.35,0.25) -- (3.65,0.25); 
\draw [<->] (2.35,0) -- (3.65,0); 
\draw [<->] (2.35,-0.25) -- (3.65,-0.25); 
\draw [<->] (0.35,0) -- (1.65,0);

\fill[color=white] (0,0) circle (.3);
\fill[color=white] (2,0) circle (.3);
\fill[color=white] (4,0) circle (.3);

\draw (0,0) node {$W_0$};
\draw (2,0) node {$W_1$};
\draw (4,0) node {$W_2$};
\draw (1,0.10) node {\tiny $s_3$};
\draw (3,0.10) node {\tiny $s_2$};
\draw (3,0.35) node {\tiny $s_1$};
\draw (3,-0.15) node {\tiny $s_4$};

\end{tikzpicture}
\end{center} 
That is, the morphisms in $\mathcal{U}$ between these objects are equivalence classes of directed walks in this multi-graph, where the equivalence relation on walks is that given by the Coxeter relations. For example, the three non-identity elements of $U_0$, namely $s_3s_1s_2s_3$, $s_3s_2s_4s_3$, and $s_3s_1s_4s_3$, are walks from $W_0$ to $W_2$ and back again. The Coxeter relations imply that, for example, walking from $W_1$ to $W_2$ along the $s_1$ edge and then walking back along the $s_2$ edge is equivalent to walking first along the $s_2$ edge and then back along the $s_1$ edge; and walking from $W_1$ to $W_2$ and back along the same edge is equivalent to not moving.   
\end{example}

\begin{remark}\label{rem:brinkhowlett}
One can give an alternative description of the groupoid $\mathcal{U}$ in terms of the root system $\Phi$ of $(W,S)$, as in Remark~\ref{rem:roots}. The map sending $W_0$ to the subset $\Pi_0\subseteq\Phi^+$ is a bijection between reflection subgroups of $W$ and subsets of $\Phi^+$ satisfying the condition of~\cite[Theorem 4.4]{DYER_SUBGROUPS}; call these the \emph{simple subsets} of $\Phi^+$. For $w\in W$, we have $N(w^{-1})\cap W_1=\emptyset$ if and only if $w(\Pi_1)$ is a subset of $\Phi^+$, in which case it is clearly a simple subset. Hence $\mathcal{U}$ is isomorphic to the groupoid $\mathcal{U}'$ where the objects are simple subsets of $\Phi^+$ and $$\Hom_{\mathcal{U}'}(\Pi_1,\Pi_2):=\{w\in W;w(\Pi_1)=\Pi_2\}.$$ Note that, under this isomorphism, the generators of $\mathcal{U}$ described in Lemma~\ref{lem:presentation2} correspond to the morphisms $s(\Pi_1)\overset{s}{\longleftarrow}\Pi_1$ where $s\in S$ and $\alpha_s\notin\Pi_1$.

In~\cite{BRINKHOWLETT}, Brink and Howlett effectively study the full sub-groupoid $\mathcal{U}''$ of $\mathcal{U}'$ where the objects are the subsets of $\Pi$, all of which are simple in the above sense; the corresponding reflection subgroups are the standard parabolic subgroups of $W$. In fact, they restrict attention to the connected component $G(J,W)$ of $\mathcal{U}''$ containing a fixed subset $J\subseteq\Pi$, and give a presentation of $G(J,W)$ in~\cite[Theorem A]{BRINKHOWLETT} and a `semidirect product decomposition' of $G(J,W)$ in~\cite[Theorem B]{BRINKHOWLETT}. Since $G(J,W)$ has fewer objects in general than the connected component of $\mathcal{U}'$ which contains it, their presentation is both more complicated than that given in Lemma~\ref{lem:presentation2}, and more useful as a way of describing the endomorphism groups $U_0$. However, their results say nothing about the connected components of $\mathcal{U}$ consisting of non-parabolic reflection subgroups.     
\end{remark}

\subsection{Artin groupoids}
As in the previous section, let $B$ denote the Artin group associated to $(W,S)$, and let $\pi:B\onto W$ be the projection with its non-homomorphic section $w\mapsto\underline{w}$. Define a groupoid $\widehat{\mathcal{B}}$ with the same set of objects as $\mathcal{N}$, but with
$$
\Hom_{\widehat{\mathcal{B}}}(W_1,W_2):=\{\beta\in B; \pi(\beta)W_1\pi(\beta)^{-1}=W_2\}.
$$
For any object $W_0$ we have $\End_{\widehat{\mathcal{B}}}(W_0)=\pi^{-1}(N_W(W_0))=\widehat{B}_0$. It is easy to see that the subgroups $K_0\vartriangleleft \widehat{B}_0$ defined in the previous section constitute a compatible collection, so that we can form the quotient groupoid $\widetilde{\mathcal{B}}:=\widehat{\mathcal{B}}/K_\bullet$ with $\End_{\widetilde{\mathcal{B}}}(W_0)=\widetilde{B}_0$.  

The projection $\pi:B\onto W$ induces a full functor $\widehat{\mathcal{B}}\onto\mathcal{N}$. Since $K_0<\ker(\pi)$, this functor factors through a full functor $\widetilde{\Pi}:\widetilde{\mathcal{B}}\onto\mathcal{N}$. By definition, $\widetilde{\Pi}$ is the identity on objects and induces the projections $\tilde{\pi}_0:\widetilde{B}_0\onto N_W(W_0)$ on endomorphism groups.

The point now is that $\mathcal{U}$ can be embedded in $\widetilde{\mathcal{B}}$ by taking positive lifts:

\begin{theorem} \label{thm:splitCox-groupoid}
There is a faithful functor $\Psi:\mathcal{U}\into\widetilde{\mathcal{B}}$ which is the identity on objects and maps each morphism $w\in\Hom_{\mathcal{U}}(W_1,W_2)$ to $\underline{w}K_1\in\Hom_{\widetilde{\mathcal{B}}}(W_1,W_2)$. We have the following commutative diagram of functors.
$$
\xymatrix{
 & \widetilde{\mathcal{B}} \ar@{->>}[d]^{\widetilde{\Pi}} \\
 \mathcal{U} \ar@{-->}[ur]^{\Psi}\ar@{^{(}->}[r] \ar[dr]_{\sim} & \mathcal{N}\ar@{->>}[d]\\
 & \overline{\mathcal{N}}}
$$
\end{theorem}

\begin{proof}
Recall the presentation of $\mathcal{U}$ given in Lemma~\ref{lem:presentation2}. We first want to show that there exists a functor $\Psi:\mathcal{U}\to\widetilde{\mathcal{B}}$ which is the identity on objects and maps each generating morphism $sW_1s\overset{s}{\longleftarrow} W_1$ of $\mathcal{U}$ to the morphism $sW_1s\overset{\sigma K_1}{\longleftarrow} W_1$ of $\widetilde{\mathcal{B}}$, where $\sigma$ is the Artin generator of $B$ corresponding to $s\in S$. We need only check that these morphisms in $\widetilde{\mathcal{B}}$ satisfy the required Coxeter relations. The relation
\[
(W_1\overset{\sigma (\sigma K_1\sigma^{-1})}{\longleftarrow} sW_1s\overset{\sigma K_1}{\longleftarrow} W_1)=\mathrm{id}_{W_1}
\] 

is equivalent to $\sigma^2\in K_1$, which holds because $s\notin W_1$ by assumption. The braid-type relations hold simply because the analogous braid relations hold in $B$.

Now it is clear that the functor $\Psi:\mathcal{U}\to\widetilde{\mathcal{B}}$ defined in this way maps each morphism $W_2\overset{w}{\longleftarrow} W_1$ of $\mathcal{U}$ to the morphism $W_2\overset{\underline{w}K_1}{\longleftarrow} W_1$ of $\widetilde{\mathcal{B}}$. The top commutative triangle follows, and implies that $\Psi$ is faithful. The bottom commutative triangle is from Lemma~\ref{lem:semi-direct-groupoid}.
\end{proof}

Restricting the functors in Theorem~\ref{thm:splitCox-groupoid} to the endomorphism groups at a particular object $W_0$, we deduce the following commutative diagram of groups.
\begin{equation}
\xymatrix{
 & \widetilde{B}_0 \ar@{->>}[d]^{\tilde{\pi}_0} \\
 U_0 \ar@{-->}[ur]^{\psi}\ar@{^{(}->}[r] \ar[dr]_{\sim} & N_W(W_0)\ar@{->>}[d]\\
 & N_W(W_0)/W_0}
\end{equation}
Thus, we have a new proof of the existence of the homomorphism $\psi:U_0\into\widetilde{B}_0$ which splits the short exact sequence~\eqref{eq:ses} as in Theorem~\ref{thm:splitfiniteCox} (still under the assumption~\eqref{eq:assumption}, so that this short exact sequence exists).

\begin{remark}
As noted in Remark~\ref{rem:brinkhowlett}, a particularly interesting sub-groupoid of $\mathcal{U}$ is the Brink--Howlett groupoid obtained by restricting to standard parabolic subgroups of $W$. As a consequence of Theorem~\ref{thm:splitCox-groupoid} we have an embedding of this Brink--Howlett groupoid in $\widetilde{\mathcal{B}}$ also.   
\end{remark}

\section{An example of splitting for the group $G(d,1,n)$}
\label{sec:gd1n}

In this section, we prove that the short exact sequence~\eqref{eq:ses} splits when $W$ is the complex reflection group $G(d,1,n)$ and $W_0$ is the parabolic subgroup $G(d,1,k)$ for $1\leq k\leq n-1$. We will see in Example~\ref{ex:g312} that this splitting does not hold for arbitrary parabolic subgroups $W_0<G(d,1,n)$. 

\subsection{Preliminaries}\label{notation_gd1n}

Fix $n\geq 1$, $d\geq 2$. Recall that the group $W=G(d,1,n)$ has a Coxeter-like presentation with generating set $S=\{t_1, s_1, \cdots, s_{n-1}\}$, with relations given by the type $B_{n}$ braid relations (with $t_1 s_1 t_1 s_1=s_1 t_1 s_1 t_1$), the relation $t_1^d=1$, and the relations $s_i^2=1$ for all $1\leq i \leq n-1$. This presentation is encapsulated in the following diagram from~\cite{BMR}. 
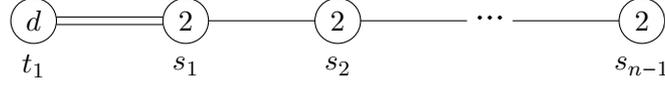
\begin{center}
\begin{tikzpicture}
\draw [-] (0.3,0.05) -- (1.7,0.05);
\draw [-] (0.3,-0.05) -- (1.7,-0.05);
\draw [-] (2.3,0) -- (3.7,0);
\draw [-] (4.3,0) -- (5.7,0); 
\draw [-] (6.3,0) -- (7.7,0);
\draw (0,0) circle (.3);
\draw (2,0) circle (.3);
\draw (4,0) circle (.3);
\draw (8,0) circle (.3);
\draw (0,0) node {$d$};
\draw (2,0) node {$2$};
\draw (4,0) node {$2$};
\draw (6,0) node {$\cdots$};
\draw (8,0) node {$2$};
\draw (0,-.6) node {$t_1$};
\draw (2,-.6) node {$s_1$};
\draw (4,-.6) node {$s_2$};
\draw (8,-.6) node {$s_{n-1}$};
\end{tikzpicture}
\end{center} 

In the standard realization of $W$ as the group of monomial $n\times n$ matrices whose nonzero entries are $d$th roots of unity, $t_1$ is the diagonal matrix with $\exp(2\pi \sqrt{-1}/d)$ in the first diagonal entry and $1$ in the other diagonal entries, and $s_1, s_2, \cdots, s_{n-1}$ are the standard permutation matrices for the adjacent transpositions.  

Note that when $d=2$ we recover the Coxeter group of type $B_n$. By~\cite[Theorem 3.6]{BMR}, the braid group $B$ of $W$ can be identified with the Artin group of type $B_n$, whatever the value of $d$. We denote its standard Artin generating set by $\Sigma=\{\tau_1, \sigma_1, \cdots, \sigma_{n-1}\}$, where $\pi(\tau_1)=t_1$ and $\pi(\sigma_i)=s_i$; these Artin generators are braided reflections in the sense of Section~\ref{sec:definitions}. Then $P=\ker(\pi:B\onto W)$ is generated by elements of two types: the first type of generator is $\beta\tau_1^d\beta^{-1}$ for some $\beta\in B$, which topologically is a meridian around the hyperplane for the order-$d$ reflection $\pi(\beta\tau_1\beta^{-1})$, and the second type of generator is $\beta\sigma_i^2\beta^{-1}$ for $\beta\in B$ and $1\leq i\leq n-1$, which topologically is a meridian around the hyperplane for the order-$2$ reflection $\pi(\beta\sigma_i\beta^{-1})$.

A major difference between the $d=2$ Coxeter case and the $d\geq 3$ non-Coxeter case is that in the latter case there is no natural way to define a positive lifting map $W\to B$. However, consider the reflections $t_i:=s_{i-1} \cdots s_1 t_1 s_1 \cdots s_{i-1}$ for $2\leq i\leq n$. In matrix terms, $t_i$ is the diagonal matrix with $\exp(2\pi \sqrt{-1}/d)$ in the $i$th diagonal entry and $1$ in the other diagonal entries. In the $d=2$ case, $s_{i-1} \cdots s_1 t_1 s_1 \cdots s_{i-1}$ is a reduced expression, so the positive lift of $t_i$ is $\sigma_{i-1} \cdots \sigma_1 \tau_1 \sigma_1 \cdots \sigma_{i-1}$. This motivates defining the `positive lift' $\tau_i:=\sigma_{i-1} \cdots \sigma_1 \tau_1 \sigma_1 \cdots \sigma_{i-1}$ for arbitrary $d$.

\subsection{A direct product decomposition}\label{no_type_a}

Now fix $1\leq k\leq n-1$ and let $S_0=\{t_1, s_1, \cdots, s_{k-1}\}$ and $W_0=\langle S_0\rangle\cong G(d,1,k)$. From the matrix realization it is easy to see that 
\begin{equation} \label{eq:direct}
N_W(W_0)=W_0\times U_0
\end{equation}
where $U_0:=\langle S_U\rangle\cong G(d,1,n-k)$ for $S_U:= \{ t_{k+1}, s_{k+1}, s_{k+2}, \dots, s_{n-1}\}$. 
The notation $U_0$ is intended to be reminiscent of the Coxeter case, and indeed when $d=2$ this subgroup $U_0$ does coincide with that in Lemma~\ref{lem:semi-direct}; the semi-direct product happens to be direct in this case. Since both $W_0$ and $U_0$ are groups of the same form as $W$, the above comments about $W$ apply also to them with the obvious modifications.

The group $K_0=\ker(P\onto P_0)$ is generated by those generators $\beta\tau_1^d\beta^{-1}$ and $\beta\sigma_i^2\beta^{-1}$ of $P$ for which the corresponding hyperplane is not in $\mathcal{A}_0$, i.e.\ for which the reflection $\pi(\beta\tau_1\beta^{-1})$ or $\pi(\beta\sigma_i\beta^{-1})$ does not belong to $W_0$.

Since $W_0$ is a parabolic subgroup of $W$, we have an injective homomorphism $B_0\into B$ as in~\eqref{eq:parabolic} whose image is a complement to $K_0$ in $\pi^{-1}(W_0)$. In this case, the homorphism is the obvious one from the Artin group of type $B_k$ to the Artin group of type $B_n$, sending $\tau_1$ to $\tau_1$ and $\sigma_j$ to $\sigma_j$ for $1\leq j\leq k-1$. So the inclusion $B_0\into\widetilde{B}_0$ maps $\tau_1$ to $\tau_1K_0$ and $\sigma_j$ to $\sigma_jK_0$ for $1\leq j\leq k-1$. 

We can now prove an analogue of Theorem~\ref{thm:splitfiniteCox} in the present case, where the splitting is still in some sense given by taking positive lifts.

\begin{proposition}\label{splittings_gd1n}
With $W=G(d,1,n)$ and $W_0=G(d,1,k)$ as above, the splitting of the short exact sequence~\eqref{eq:ses-triv} given by~\eqref{eq:direct} lifts to a splitting of the short exact sequence~\eqref{eq:ses}. Namely, after identifying $N_W(W_0)/W_0$ with $U_0$, the splitting of~\eqref{eq:ses} is an injective group homomorphism $\psi:U_0\into\widetilde{B}_0$ which is defined on the generating set $S_U$ by 
\[ \psi(t_{k+1})=\tau_{k+1} K_0\ \text{and}\ \psi(s_i)= \sigma_i K_0,\ \text{for}\ k+1\leq i\leq n-1. \] 
We have a direct product decomposition
$$
\widetilde{B}_0 \cong B_0 \times \psi(U_0).
$$

\end{proposition} 

\begin{proof}
We first want to show that there exists a group homomorphism $\psi:U_0\to\widetilde{B}_0$ which has the stated definition on the generators. For this, we must show that the elements $\tau_{k+1} K_0, \sigma_{k+1}K_0,\cdots,\sigma_{n-1}K_0$ of $\widetilde{B}_0$ satisfy the relations in the Coxeter-like presentation of $U_0\cong G(d,1,n-k)$ analogous to that given above for $W$. 

For the braid relations in this presentation, we can in fact see that they hold already for the elements $\tau_{k+1},\sigma_{k+1},\cdots,\sigma_{n-1}$ in the type-$B_n$ Artin group $B$. This is clear for the braid relations not involving $\tau_{k+1}$, since those are themselves relations in the Artin presentation of $B$. Hence we only need to check that 
\begin{equation} \label{eq:artin-relations}
\begin{split}
\tau_{k+1}\sigma_i&=\sigma_i\tau_{k+1},\ \text{for}\ k+2\leq i\leq n,\ \text{and}\\
\tau_{k+1} \sigma_{k+1} \tau_{k+1} \sigma_{k+1}&=\sigma_{k+1} \tau_{k+1} \sigma_{k+1} \tau_{k+1}.
\end{split}
\end{equation} 
Note that the truth of~\eqref{eq:artin-relations} is independent of $d$, so we can temporarily assume that we are in the $d=2$ Coxeter case. Then~\eqref{eq:artin-relations} follows from the observations that
\begin{equation} \label{eq:coxeter-relations}
\begin{split}
t_{k+1}s_i&=s_it_{k+1},\ \text{for}\ k+2\leq i\leq n,\\
t_{k+1}s_{k+1}t_{k+1}s_{k+1}&=
s_{k+1}t_{k+1} s_{k+1}t_{k+1},
\end{split}
\end{equation}
and moreover that the lengths add in each of the expressions in~\eqref{eq:coxeter-relations}.

It remains to show that the order relations hold in $\widetilde{B}_0$. For $k+1\leq i\leq n-1$ we have $s_i\notin W_0$, so $\sigma_i^2\in K_0$ which means that $(\sigma_i K_0)^2=1K_0$ as required. We also need to show that $(\tau_{k+1}K_0)^d=1K_0$. Note that
\begin{equation} \label{eq:factorization}
\tau_{k+1}=(\sigma_{k} \cdots \sigma_{1} \tau_1 \sigma_{1}^{-1} \cdots \sigma_{k}^{-1})(\sigma_{k} \cdots \sigma_{2} \sigma_{1}^2 \sigma_{2}^{-1} \cdots \sigma_{k}^{-1})(\sigma_{k} \cdots \sigma_{3} \sigma_{2}^2 \sigma_{3}^{-1} \cdots \sigma_{k}^{-1})\cdots \sigma_{k}^2.
\end{equation}
Now all the factors on the right-hand side of~\eqref{eq:factorization} except the first factor belong to $K_0$, since for all $1\leq i \leq k$, the reflection $s_k\cdots s_i \cdots s_k$ has order $2$ and is not in $W_0$. Hence 
\begin{equation}
(\tau_{k+1} K_0)^d=(\sigma_{k} \cdots \sigma_{1} \tau_1 \sigma_{1}^{-1} \cdots \sigma_{k}^{-1})^d K_0=\sigma_{k} \cdots \sigma_{1} \tau_1^d \sigma_{1}^{-1} \cdots \sigma_{k}^{-1}=1K_0,
\end{equation} 
where the last equation holds since $s_k\cdots s_1 t_1 s_1 \cdots s_k=t_{k+1}$ is a reflection of order $d$ which is not in $W_0$. This concludes the proof that the homomorphism $\psi:U_0\to\widetilde{B}_0$ exists. 

Since $\tilde{\pi}_0(\psi(u))=u$ holds when $u$ is one of the generators of $U_0$, it holds for all $u\in U_0$. Hence $\psi:U_0\into\widetilde{B}_0$ is injective and is a splitting of~\eqref{eq:ses}.

Finally, to show that the semidirect product $B_0\rtimes\psi(U_0)$ is direct, we need to show that each of $\tau_1 K_0,\sigma_1 K_0,\cdots, \sigma_{k-1} K_0$ commutes with each of $\tau_{k+1}K_0,\sigma_{k+1}K_0,\cdots,\sigma_{n-1}K_0$ in $\widetilde{B}_0$. Since each of $\tau_1,\sigma_1,\cdots, \sigma_{k-1}$ commutes with each of $\sigma_{k+1},\cdots,\sigma_{n-1}$ as part of the Artin relations of $B$, it suffices to show that $\tau_1\tau_{k+1}=\tau_{k+1}\tau_1$ and $\sigma_i\tau_{k+1}=\tau_{k+1}\sigma_i$ for $0\leq i\leq k-1$. These equations can be proved in the same way as~\eqref{eq:artin-relations}.

\end{proof}


\section{Counter-examples in the general case}
\label{sec:counterex}

In this section we demonstrate that the short exact sequence~\eqref{eq:ses} need not split in the general setting of Section~\ref{sec:definitions}, when $W$ is a finite complex reflection group. It is notable that in some of our counter-examples the group $W$ is close to being a Coxeter group, in the sense that it is a Shephard group, or in the sense that all its reflections have order $2$; nevertheless the relationship between $W$ and its subgroup $W_0$ fails to be sufficiently like the Coxeter case. 

\subsection{Non-parabolic reflection subgroups with no complement in their normalizer}
\label{ss:no-complement}

As mentioned in the introduction, it was shown by Muraleedaran and Taylor~\cite{TAYLORNORM} that when $W_0$ is a parabolic subgroup of $W$, there is always a subgroup of $N_W(W_0)$ which is complementary to $W_0$. There are cases where $W_0$ is a non-parabolic reflection subgroup and there is no such complement, i.e.\ the short exact sequence~\eqref{eq:ses-triv} does not split. This of course rules out the splitting of~\eqref{eq:ses}.
\begin{example} \label{ex:cyclic}
Suppose that $W=\langle s\rangle$ is cyclic of order $d\geq 2$. If $e$ is a divisor of $d$ with $1<e<d$, then $W_0=\langle s^e\rangle$ is a non-parabolic reflection subgroup of $W$ of order $d/e$. If $\gcd(e,d/e)>1$, there is clearly no complement to $W_0$ in $W$.
\end{example}
One could eliminate Example~\ref{ex:cyclic} by restricting to the case when $W_0$ is a \emph{full} reflection subgroup of $W$, meaning that $W_0$ contains all the reflections in $W$ whose hyperplane is in $\mathcal{A}_0$. However, this still leaves many examples; we content ourselves with two.
\begin{example} \label{ex:g422}
Let $W$ be the rank-$2$ imprimitive irreducible reflection group $G(4,2,2)$ of order $16$, in which the reflections are the order-$2$ unitary reflections of $\C^2$ with hyperplanes defined by the linear forms $z_1,z_2,z_1+z_2,z_1-z_2,z_1+\sqrt{-1}\,z_2,z_1-\sqrt{-1}\,z_2$. Let $W_0$ be the reflection subgroup of order $4$ generated by the reflections with hyperplanes defined by $z_1,z_2$. Then $N_W(W_0)=W$, but there is no complement to $W_0$ in $W$. Indeed, if there was such a complement $U_0$, then $U_0\cap G(4,4,2)$ would be a complement to the centre in a dihedral group of order $8$. 
\end{example}
\begin{example} \label{ex:g6}
Let $W$ be the rank-$2$ primitive irreducible reflection group of order $48$ known as $G_6$ in the Shephard-Todd numbering (a Shephard group). The reflection subgroup $W'$ of $W$ generated by the six reflections of order $2$ is a copy of $G(4,2,2)$. If we let $W_0\vartriangleleft W'$ be as in the previous example, then $N_W(W_0)=W'$, so once again there is no complement to $W_0$ in $N_W(W_0)$.
\end{example}

\subsection{Central elements of braid groups}
\label{ss:central}

Even when~\eqref{eq:ses-triv} splits, there can be an obstruction to the splitting of~\eqref{eq:ses} coming from the centre $Z(B)$ of $B$. This  obstruction can be seen already in the most trivial non-parabolic example.
\begin{example} \label{ex:cyclic2}
Continue the notation of Example~\ref{ex:cyclic}, without assuming $\gcd(e,d/e)>1$. Then $\widetilde{B}_0=B=\langle\sigma\rangle$ is infinite cyclic and $B_0$ is the nontrivial subgroup $\langle\sigma^e\rangle$, so~\eqref{eq:ses} does not split, regardless of whether~\eqref{eq:ses-triv} splits.  
\end{example} 

We now show that a similar phenomenon happens more generally, including in some cases when $W_0$ is a parabolic subgroup of $W$.

We assume henceforth that $W$ is irreducible. The centre $Z(W)$ is then cyclic and acts on $\C^n$ by scalar multiplication. Let $d:=|Z(W)|$ and let $z_W$ denote the generator of $Z(W)$ which acts on $\C^n$ as multiplication by $\exp(2\pi \sqrt{-1}/d)$.

Recall from~\cite[Lemma 2.4]{BMR} that there is a canonical central element $z_P\in P=\pi_1(X,\tilde{x})$, the homotopy class of the loop $[0,1]\to X:t\mapsto \exp(2\pi \sqrt{-1}\, t)\tilde{x}$. Define $z_{P_0}\in P_0=\pi_1(X^0,\tilde{x})$ similarly, as the homotopy class of the very same loop. Then under our identification of $P_0$ with $P/K_0$, $z_{P_0}$ corresponds to $z_P K_0$.
 
Let $z_B\in B=\pi_1(X/W,[\tilde{x}]_W)$ be the homotopy class of the loop in $X/W$ which is the image of the path $[0,1]\to X:t\mapsto \exp(2\pi \sqrt{-1}\, t/d)\tilde{x}$ from $\tilde{x}$ to $z_W(\tilde{x})$. It is shown in~\cite[Lemma 2.22]{BMR} that $z_B\in Z(B)$, that $\pi(z_B)=z_W$ and that $z_B^d=z_P$. (In fact, it is known that $Z(B)=\langle z_B\rangle$, by~\cite[Theorem 2.24]{BMR} and~\cite[Theorem 12.8]{BESSISKPI1}; we will not need this.)  
 
\begin{proposition} \label{prop:central}
Suppose that $W$ is irreducible with $d=|Z(W)|$ and that $W_0$ is a reflection subgroup of $W$.
\begin{enumerate}
\item If~\eqref{eq:ses} splits, then there is an element $\gamma\in B_0$ such that $\gamma^d=z_{P_0}$.
\item If $N_W(W_0)=W_0\times Z(W)$, then the converse to (1) also holds.
\end{enumerate}
\end{proposition} 
\begin{proof}
Note that $z_B K_0\in\widetilde{B}_0$ maps to $z_W W_0\in N_W(W_0)/W_0$ under the homomorphism in~\eqref{eq:ses}. 
Since $(z_W W_0)^d=1W_0$ holds in $N_W(W_0)/W_0$, the splitting of~\eqref{eq:ses} implies that for some $\gamma\in B_0$ we have $(\gamma^{-1}(z_B K_0))^d=1K_0$ in $\widetilde{B}_0$. Moreover, if $N_W(W_0)=W_0\times Z(W)$, then the splitting of~\eqref{eq:ses} is equivalent to the existence of such $\gamma\in B_0$. Since $z_B K_0\in Z(\widetilde{B}_0)$ and $(z_B K_0)^d=z_P K_0=z_{P_0}$, the equation $(\gamma^{-1}(z_B K_0))^d=1K_0$ in $\widetilde{B}_0$ is equivalent to the equation $\gamma^d=z_{P_0}$ in $B_0$.
\end{proof}

In the special case when $W_0=\langle s\rangle$ is cyclic of order $m\geq 2$, we have that $B_0=\langle\sigma\rangle$ is infinite cyclic with $\sigma^m=z_{P_0}$, so the existence of $\gamma\in B_0$ such that $\gamma^d=z_{P_0}$ is equivalent to $d\mid m$. This means that if $d\nmid m$, Proposition~\ref{prop:central}(1) guarantees that~\eqref{eq:ses} does not split. Example~\ref{ex:cyclic2} was such a case, and we can now easily find similar non-splitting examples with $W_0$ parabolic. The examples below all have the property that $N_W(W_0)=W_0\times Z(W)$.
\begin{example} \label{ex:g312}
Let $W=G(3,1,2)$, for which $d=3$, and let $W_0$ be a rank-$1$ parabolic subgroup generated by a reflection of order $2$. Then~\eqref{eq:ses} does not split. 
\end{example}
\begin{example}
Let $W=G(4,2,2)$ as in Example~\ref{ex:g422}, for which $d=4$, and let $W_0$ be any rank-$1$ parabolic subgroup, necessarily of order $2$. Then~\eqref{eq:ses} does not split. 
\end{example}
\begin{example}
Let $W$ be the rank-$2$ primitive irreducible reflection group of order $24$ known as $G_4$ in the Shephard-Todd numbering, for which $d=2$. It is a Shephard group, with braid group the Artin group of type $A_2$. Let $W_0$ be a rank-$1$ parabolic subgroup, necessarily of order $3$. Then~\eqref{eq:ses} does not split. 
\end{example}

Much is known about the existence of roots of the canonical element $z_P$ in the braid group $B$. In particular, if $W$ is irreducible and \emph{well-generated}, meaning that it can be generated by $n$ reflections, then Bessis proved in~\cite[Theorem 12.4(i)]{BESSISKPI1} that for any positive integer $m$, there exists an element $\gamma\in B$ such that $\gamma^m=z_P$ if and only if $m$ is \emph{regular} for $W$ (the ``if'' direction is easy). Recall that the regular numbers for $W$ are the orders of those roots of unity which arise as eigenvalues for elements of $W$ where the corresponding eigenvector belongs to the hyperplane complement $X$. The regular numbers can be deduced from the degrees and codegrees of $W$ as explained in~\cite[Theorem 1.9(1)]{BESSISKPI1}.

So Proposition~\ref{prop:central} implies the following result.
\begin{corollary} \label{cor:regular}
Suppose that $W$ is irreducible with $d=|Z(W)|$, and that $W_0$ is a reflection subgroup of $W$ such that each irreducible constituent of $W_0$ is well-generated. 
\begin{enumerate}
\item If~\eqref{eq:ses} splits, then $d$ is regular for each irreducible constituent of $W_0$.
\item If $N_W(W_0)=W_0\times Z(W)$, then the converse to (1) also holds.
\end{enumerate}
\end{corollary}

\bibliography{normartin_biblio}

\begin{bibdiv}
\begin{biblist}

\bib{BESSISKPI1}{article}{
      author={Bessis, D.},
       title={{F}inite complex reflection arrangements are ${K}(\pi,1)$},
        date={2015},
     journal={Ann.\ of Math.\ (2)},
      volume={181},
      number={3},
       pages={809 \ndash  904},
}

\bib{BOURBAKI}{book}{
      author={Bourbaki, N.},
       title={{\'E}l\'ements de math\'ematique. {F}asc. {XXXIV}. {G}roupes et
  alg\`ebres de {L}ie},
   publisher={Actualit\'es Scientifiques et Industrielles, No.~1337 Hermann,
  Paris},
        date={1968},
}

\bib{BRIESKORN}{article}{
      author={Brieskorn, E.},
       title={Die {F}undamentalgruppe des {R}aumes der regulären {O}rbits
  einer endlichen komplexen {S}piegelungsgruppe},
        date={1971},
     journal={Invent.\ Math.},
      volume={12},
       pages={57\ndash 61},
}

\bib{BRINKHOWLETT}{article}{
      author={Brink, R.},
      author={Howlett, R.~B.},
       title={{N}ormalizers of parabolic subgroups in {C}oxeter groups},
        date={1999},
     journal={Invent. Math.},
      volume={136},
      number={2},
       pages={323\ndash 351},
}

\bib{BMR}{article}{
      author={Broué, M.},
      author={Malle, G.},
      author={Rouquier, R.},
       title={Complex reflection groups, braid groups, {H}ecke algebras},
        date={1998},
     journal={J.\ Reine Angew.\ Math.},
      volume={500},
       pages={127\ndash 190},
}

\bib{DYER_SUBGROUPS}{article}{
      author={Dyer, M.},
       title={Reflection subgroups of {C}oxeter systems},
        date={1990},
     journal={J.\ Algebra},
      volume={135},
      number={1},
       pages={57\ndash 73},
}

\bib{HOWLETT}{article}{
      author={Howlett, R.~B.},
       title={Normalizers of parabolic subgroups of reflection groups},
        date={1980},
     journal={J. London Math. Soc. (2)},
      volume={21},
      number={1},
       pages={62\ndash 80},
}

\bib{MACLANE}{book}{
      author={MacLane, S.},
       title={Categories for the working mathematician},
   publisher={Graduate Texts in Mathematics, Vol. 5, 2nd edition,
  Springer-Verlag, New York--Berlin},
        date={1978},
}

\bib{YH1}{article}{
      author={Marin, I.},
       title={Artin groups and {Y}okonuma--{H}ecke algebras},
        date={2018},
     journal={Int.\ Math.\ Res.\ Not.\ IMRN},
      volume={2018},
      number={13},
       pages={4022\ndash 4062},
}

\bib{YH2}{article}{
      author={Marin, I.},
       title={Lattice extensions of {H}ecke algebras},
        date={2018},
     journal={J.\ Algebra},
      volume={503},
       pages={104\ndash 120},
}

\bib{TAYLORNORM}{article}{
      author={Muraleedaran, K.},
      author={Taylor, D.~E.},
       title={Normalisers of parabolic subgroups in finite unitary reflection
  groups},
        date={2018},
     journal={J.\ Algebra},
      volume={504},
       pages={479\ndash 505},
}

\bib{SALVETTI}{article}{
      author={Salvetti, M.},
       title={The homotopy type of {A}rtin groups},
        date={1994},
     journal={Math. Res. Lett.},
      volume={1},
      number={5},
       pages={565\ndash 577},
}

\bib{VANDERLEK}{thesis}{
      author={van~der Lek, H.},
       title={The homotopy type of complex hyperplane complements},
        type={Ph.D. Thesis},
        date={1983},
}

\end{biblist}
\end{bibdiv}

\end{document}